\documentclass{amsart}

\usepackage{amsmath, amsthm, mathtools}
\usepackage{amssymb}
\usepackage{subfiles}
\usepackage{todonotes}
\usepackage{verbatim}
\usepackage{enumitem}
\usepackage{url}
\usepackage[T1]{fontenc}
\usepackage{hyperref}
\hypersetup{colorlinks=true}
\definecolor{navyblue}{rgb}{0.0, 0.0, 0.5}
\hypersetup{colorlinks=true, allcolors = navyblue}

\usepackage{cite}

\theoremstyle{definition}

\newtheorem{thm}{Theorem}[section]
\newtheorem*{thm*}{Theorem}
\newtheorem{lemma}[thm]{Lemma}

\newtheorem{defn}[thm]{Definition}
\newtheorem{claim}[thm]{Claim}
\newtheorem{prop}[thm]{Proposition}
\newtheorem{cor}[thm]{Corollary}

\newtheorem{remark}[thm]{Remark}
\newtheorem{fact}[thm]{Fact}
\newtheorem{question}[thm]{Question}
\newtheorem{ex}[thm]{Example}

\newtheorem{obs}[thm]{Observation}

\renewcommand{\subset}{\subseteq}
\renewcommand{\supset}{\supseteq}

\newcommand\force{\Vdash}
\newcommand\N{\mathbb{N}}
\newcommand\Z{\mathbb{Z}}
\newcommand\Q{\mathbb{Q}}
\newcommand\R{\mathbb{R}}

\newcommand\GL{\mathrm{GL}}
\newcommand\analytic{\boldsymbol{\Sigma}^1_1}
\DeclareMathOperator{\spann}{span}
\DeclareMathOperator{\Ar}{Ar}
\DeclareMathOperator{\LO}{LO}
\DeclareMathOperator{\Inn}{Inn}

\DeclareMathOperator{\tc}{tc}
\DeclareMathOperator{\Hom}{Hom}
\DeclareMathOperator{\Homeo}{Homeo}
\DeclareMathOperator{\Aut}{Aut}

\DeclareMathOperator{\dom}{dom}

\newcommand{\set}[2]{ \left\{ #1 :\, #2 \right\} }
\newcommand{\seqq}[2]{ \left\langle #1 :\, #2\right\rangle }

\newcommand{\embeds}{\sqsubseteq}

\makeatletter
\@namedef{subjclassname@2020}{%
  \textup{2020} Mathematics Subject Classification}
\makeatother

\begin{document}

\title[]{Anti-classification results for groups acting freely on the line}
\date{\today}

\author[F.~Calderoni]{Filippo Calderoni}
\address{Department of Mathematics, Rutgers University, 
Hill Center for the Mathematical Sciences,
110 Frelinghuysen Rd.,
Piscataway, NJ 08854-8019}
\email{filippo.calderoni@rutgers.edu}

\author[D.~Marker]{David Marker}
\address{Department of Mathematics, Statistics, and Computer Science, University
of Illinois at Chicago, Chicago IL 60607, USA}
\email{marker@uic.edu}

\author[L.~Motto~Ros]{Luca Motto~Ros}
\address{Dipartimento di matematica \guillemotleft Giuseppe Peano\guillemotright, Universit\`a di Torino, 10121 Torino, Italy}
\email{luca.mottoros@unito.it}

\author[A.~Shani]{Assaf Shani}
\address{Department of Mathematics, Harvard University, Cambridge, MA 02138, USA}
\email{shani@math.harvard.edu}

\thanks{The third author is supported by the project PRIN~2017 ``Mathematical Logic: models, sets, computability'', prot. 2017NWTM8R}

\maketitle

\begin{abstract}

We explore countable ordered Archimedean groups from the point of view of descriptive set theory. We introduce the space of Archimedean left-orderings $\Ar(G)$ for a given countable group \(G\), and prove that the
equivalence relation induced by the natural action of $\GL_2(\Q)$ on  $\Ar(\Q^2)$ is not concretely classifiable.
Then we analyze the isomorphism relation for countable ordered Archimedean groups, and pin its complexity in terms of the hierarchy of Hjorth, Kechris and Louveau \cite{HjoKecLou}. In particular, we show that its potential class is not $\boldsymbol{\Pi}^0_3$. This topological constraint prevents classifying Archimedean groups using countable subsets of reals. We obtain analogous results for the bi-embeddability relation, and we consider similar problems for circularly ordered groups, and o-minimal structures such as ordered divisible Abelian groups, and real closed fields.
Our proofs combine classical results on Archimedean groups, the theory of Borel equivalence relations, and analyzing 
definable sets in the basic Cohen model and other models of Zermelo-Fraenkel set theory without choice.
\end{abstract}


\section{Introduction}

A group \(G\) is left-orderable if it admits a strict total order that is invariant under left multiplication.
If \(G\) is a countable left-orderable group, it is not hard to produce a faithful action of \(G\) on the real line by order preserving homeomorphism associated to any left-ordering on \(G\). This action is called dynamical realization.

Turning orders into actions gives rise to a beautiful interplay between algebraic properties and dynamical structure.  This correspondence
dates back as early as the work by H\"older in 1901, when he proved that any group acting freely on \(\R\) is Abelian. (See Ghys~\cite[Theorem~6.10]{Ghy}.)
Besides its classical relevance, the dynamical approach has been proven useful in numerous situations. For example, in the context of the fundamental groups of \(3\)-manifolds, orderability has crucial implications in the theory of foliations and laminations as showed by Calegari and Dunfield~\cite{CalDun}. More recently, Mann and Rivas~\cite{ManRiv} established a characterization of certain left-orders in terms of rigidity phenomena in the moduli space \(\Hom(G, \Homeo_{+}(\R))\).

H\"older was particularly interested in left-ordered groups having the Archimedean type property. Recall that an ordered group \((G,<)\) is \emph{Archimedean} if for all \(x,y \in G\) there is \(n\in \Z\) such that \(x <  y^n\). Archimedean groups can be characterized in terms of their dynamical realizations:
Archimedean orders correspond to free actions on the line.
(An account of these results can be found in the work of Navas~\cite[Section~3]{Nav10}.)

In this paper we study Archimedean ordered groups from the point of view of descriptive set theory. Mainly, we analyze the possibility of classifying countable Archimedean groups completely up to isomorphism.

To explain our approach we shall present briefly the main ideas behind Borel classification theory.
Given a class of mathematical structures \(\mathcal{X}\) with a corresponding notion of isomorphism \(\cong_\mathcal{X}\), a \emph{complete classification} for \(\mathcal{X}\) is an assignment \(c\colon \mathcal{X}\to I\) such that for any \(x,y  \in \mathcal{X}\) we have
\[
x\cong_\mathcal{X} y \iff c(x) = c(y).
\]
In many well-understood situations we can parametrize the elements of \(\mathcal{X}\) as the objects of a standard Borel space\footnote{By a classical theorem of Kuratowski every uncountable standard Borel space is Borel isomorphic to \(\R\).} and we can regard \(\cong_\mathcal{X}\) as a Borel set. More precisely, we say that an equivalence relation on a standard Borel space \( X \) is \emph{Borel} if it is a Borel subset of the product space \(X\times X\). As explained in the excellent survey of Foreman~\cite{For18}, for those equivalence relations that are Borel, we can determine whether two elements are equivalent with a countable amount of data. On the other hand, when we prove that a certain equivalence relation is not Borel, we can take it as strong evidence against a satisfactory classification theory.

The ``non-Borel test'' produced anticlassification results in  many notable situations such as \cite{Hjo02,DowMon} for the isomorphism relation on torsion-free Abelian groups, and in~\cite{ForRudWei} for the isomorphism relation on the space of measure preserving transformation of the unit interval.
However, the isomorphism relation on countable Archimedean groups is easily seen to be Borel (see Section 3), and this leaves a door open to find optimal complete invariants. 

The standard notion for comparing the complexity of different equivalence relation is \emph{Borel reducibility}. If \(E,F\) are equivalence relations  (or, more generally, quasi-orders) on the standard Borel spaces \(X\) and \(Y\) respectively, we say that \(E\) is \emph{Borel reducible} to \(F\) (in symbols, \(E\leq_B F\)) if there is a Borel map \(f\colon X \to Y\) such that \(x_1\mathbin{E} x_2 \iff f(x_1)\mathbin{F} f(x_2)\). Whenever \(E \leq_B F\), then \(F\)-classes can be regarded as complete invariants for \(E\).
Moreover, any definable complete classification for \(F\) translates into a definable complete classification for \(E\). The notion of Borel reducibility has been used successfully to analyze invariants for several classification problems across mathematics. Some remarkable examples are included in~\cite{AdaKec,ForWei,Sab16,Tho03}.

As a consequence of a classical dichotomy by Silver~\cite{Sil80} the collection of Borel equivalence relations with uncountably many classes has a minimum element, the identity relation on real numbers, denoted by \(=_\R\). A Borel equivalence relation is called \emph{concretely classifiable} (or \emph{smooth}) if it is Borel reducible to \(=_\R\).
Moreover, Harrington, Kechris and Louveau~\cite{HarKecLou} showed that the equivalence relation of eventual equality \(E_0\) on \(2^\N\), the Cantor space of binary sequences equipped with the product topology, is an immediate successor of \(=_\R\).

We observe in Section~\ref{sec : Ar} that the isomorphism relation \(\cong_\mathsf{ArGp}\) on countable ordered Archimedean groups is not Borel reducible to \(=_\R\). In fact, let
 \(\cong_{\Ar(\Q^2)}\) be the order preserving isomorphism relation on the space of isomorphic copies of \(\Q^2\) with all possible Archimedean orders. 

\begin{thm}[Section~\ref{sec : Ar}]
\label{thm : Ar(Q2)}
There is a Borel reduction from \(E_0\) to \(\cong_{\Ar{(\Q^2)}}\). In particular, \(\cong_{\Ar(\Q^2)}\) is not concretely classifiable.
\end{thm}

This is an anti-classification theorem in the sense that it excludes the possibility of classifying countable Archimedean groups completely using numerical invariants, and it is also
a contribution to the analysis and classification of the left-orders on a given countable group~\cite{CC22,Cla12,Lin11,Riv12,Sik04}.
(In Section~\ref{sec : Ar}  we will point out that
Baer's analysis \cite{Bae} of torsion-free Abelian group of rank \(1\) also implies that isomorphism on countable Archimedean groups is strictly more complicated than \(=_\mathbb{R}\).)

Beyond concrete classification, one can try to assign invariants which are countable sets of reals.
For example, using the Halmos-von Neumann theory, Foreman and Louveau showed that countable sets of reals provide a successful classification of the conjugacy relation on ergodic discrete spectrum transformations (see~\cite[Section~5.2]{For}).

A natural way to capture this is by comparing our equivalence relation with the Friedman-Stanley jumps, defined as follows \cite{FriSta}.
The \emph{first Friedman-Stanley jump} is the equivalence relation $=_\R^+$ defined on the space $\mathbb{R}^{\mathbb{N}}$ by identifying two countable sequences of reals if and only if they enumerate the same set of reals. The equivalence relation $=^+_\R$ admits the natural complete classification sending $(x_n)_{n\in \N}\in\mathbb{R}^{\mathbb{N}}$ to the set of reals it enumerates, $\set{x_n}{n\in\mathbb{N}}$, so that the invariants are precisely the countable sets of reals.
A Borel equivalence relation $E$ is considered classifiable by countable sets of reals if it is Borel reducible to $=^+_\R$.

The \emph{second Friedman-Stanley jump} $=^{++}_\R$ is defined to naturally admit complete invariants which are countable sets of countable sets of reals. Similarly, for each countable ordinal $\alpha$, Friedman and Stanley defined an equivalence relation $=_\R^{\alpha+}$ whose natural complete invariants are precisely the hereditarily countable sets in $\mathcal{P}^{1+\alpha}(\mathbb{N})$, the $1+\alpha$-th iterated power set of $\mathbb{N}$.

This \textit{Friedman-Stanley hierarchy} was further refined by Hjorth, Kechris and Louveau \cite{HjoKecLou}, where they introduced equivalence relations $\cong^\ast_{\alpha+2,\beta}$ whose complexity lies between $=_\R^{\alpha+}$ and $=_\R^{(\alpha+1)+}$.
For example, they introduced two equivalence relations $\cong^\ast_{3,0}$ and $\cong^\ast_{3,1}$ such that \begin{equation*}
    {=^+_\R} <_B {\cong^\ast_{3,0}} <_B {\cong^\ast_{3,1}} <_B {=^{++}_\R}.
\end{equation*}
Moreover, it turns out~\cite{HjoKecLou} that there is a clear correspondence between the topological complexity of isomorphism relations and their possible invariants.
For example, for any isomorphism relation $E$, if $E$ is $\mathbf{\Pi}^0_3$ then it is Borel reducible to $=^+_\R$, and if $E$ is $\mathbf{\Sigma}^0_4$ then it is Borel reducible to $\cong^\ast_{3,1}$.

We show that $=^+_\R$ is Borel reducible to $\cong_\mathsf{ArGp}$ (Proposition~\ref{prop : lowbnd}), and that $\cong_\mathsf{ArGp}$ is $\mathbf{\Sigma}^0_4$, and therefore is Borel reducible to $\cong^\ast_{3,1}$ (Proposition~\ref{prop : reduction to cong_3,1}).
Next, in Section~\ref{sec : irred to 3,0} we show that \(\cong_\mathsf{ArGp}\) is genuinely more complicated than \(=^+_\R\) in a very strong sense:

\begin{thm}[{Section~\ref{sec : irred to 3,0}}]
\label{thm : main}
The isomorphism relation \(\cong_\mathsf{ArGp}\) is not Borel reducible to \(\cong_{3,0}^*\).
Consequently, countable sets of reals cannot be used to completely classify countable Archimedean groups up to isomorphism.
\end{thm}

Combining all previous results we get
\[
{=^+_\R} <_B {\cong_\mathsf{ArGp}} <_B {=^{++}_\R}.
\]
This is a curious phenomenon also from the viewpoint of the theory of Borel reducibility. To the best of our knowledge it is the first example of a natural isomorphism relation lying strictly in between two consecutive Friedman-Stanley jumps of \(=_\R\).

The proof of Theorem~\ref{thm : main} goes through the study of definable sets in models of \(\mathsf{ZF}\), the standard axioms of mathematics without the axiom of choice, and the technique of \emph{forcing}, which was developed by Cohen in 1963 to solve the long-standing Hilbert's first problem, also known as the \emph{Continuum Hypothesis}.
More specifically, working in a forcing extension of the so called ``basic Cohen model'', we construct a generic invariant for Archimedean groups which is complex enough so that it cannot be coded by any invariant for \(\cong^*_{3,0}\).
This is a striking application of the techniques recently developed  by Shani~\cite{Sha}.

Our arguments are rather flexible and can also be used to attack similar problems concerning the (bi-)embeddability relation on countable Archimedean groups using the natural generalization of Borel reducibility to quasi-orders.

\begin{thm}[{Section~\ref{sec : embeddability on ArGp}}]
\label{thm : main-embed}
Let \( \embeds_\mathsf{ArGp} \) be the embeddability relation on countable Archimedean groups, and let \( \equiv_\mathsf{ArGp} \) be the associated bi-embeddability relation. Then
\[
{=_\R^{\mathrm{cf}}} <_B {\embeds_\mathsf{ArGp}} <_B {=_\R^{2\text{-}\mathrm{cf}}}
\]
and
\[
{=_\R^+} = E_{=_\R^\mathrm{cf}} <_B {\equiv_\mathsf{ArGp}} <_B E_{=_\R^{2\text{-}\mathrm{cf}}}.
\] 
In particular, countable sets of reals cannot be used to completely classify countable Archimedean groups up to bi-embeddability.
\end{thm}

In the previous theorem \( =_\R^\mathrm{cf} \) and \( =_\R^{2\text{-}\mathrm{cf}} \) are the first and second jump of \( =_\R \), where \( P \mapsto P^\mathrm{cf} \) is the jump operator for quasi-orders introduced by Rosendal in~\cite{Ros} as the asymmetric version of the Friedman-Stanley jump \( E \mapsto E^+ \). The relations \( E_{=_\R^\mathrm{cf}} \) and \( E_{=_\R^{2\text{-}\mathrm{cf}}} \) are the equivalence relations canonically induced by the quasi-orders \( =_\R^{\mathrm{cf}} \) and \( =_\R^{2\text{-}\mathrm{cf}} \), respectively. In Section~\ref{subsec: another lower bound} we show that \(\equiv_\mathsf{ArGp}\) is not Borel reducible to \(\cong_\mathsf{ArGp}\) (Corollary~\ref{cor: biembed not orbit}).
Perhaps surprisingly, whether \(\cong_\mathsf{ArGp}\) is Borel reducible to \(\equiv_\mathsf{ArGp}\) is still open.

In Section~\ref{sec : circular orders} we analyze the isomorphism relation on the class of circularly ordered Archimedean groups. These are the algebraic counterparts of groups acting freely on the circle.
Circularly ordered groups have been recently investigated in a series of work including~\cite{BaiSam, BelClaGha, ClaManRiv, ClaGha}, and
they play a central role in the aforementioned work of Calegari and Dunfield~\cite{CalDun}.
We contrast Theorem~\ref{thm : main} by showing the following:

\begin{thm}[Section~\ref{sec : circular orders}]
\label{thm : CO}
The isomorphism and bi-embeddability relations \(\cong_\mathsf{CArGp}\) and \(\equiv_\mathsf{CArGp}\) on Archimedean circularly ordered groups are Borel bi-reducible with \(=_\R^{+}\), hence strictly simpler than their counterparts \(\cong_\mathsf{ArGp}\) and \(\equiv_\mathsf{ArGp}\).
\end{thm}

In Section~\ref{sec : ODAG} we discuss some results about the (bi-)embeddability relation between ordered groups after removing the Archimedean type property, and other ordered structures. In particular we show:

\begin{thm}[{Section~\ref{sec : ODAG}}]
\label{theorem : ODAG}
The embeddability relation on countable ordered divisible Abelian groups is a complete analytic quasi-order. Thus, the bi-embeddability relation on countable ordered divisible Archimedean groups is a complete analytic equivalence relation.
\end{thm}

This particularly implies that embeddability on countable real closed fields is a complete analytic quasi-order, and that the corresponding bi-embeddability relation is a complete analytic equivalence relation.
Theorem~\ref{theorem : ODAG} and the other results of Section 7 extend the work of Rast and Sahota~\cite{RasSah} about the isomorphism relation for first order o-minimal theories.

\section{Preliminary results}
\label{sec : Ar}

A group \(G\) is \emph{left-orderable} if it admits a strict total order
\(<\) such that \(y < z\) implies \(xy < xz\) for all \(x, y, z \in G\). 
A countable group \(G\) is left-orderable if and only if it admits a faithful action on \(\R\) by order preserving homeomorphisms.
(E.g., see \cite[Proposition~2.1]{Nav10}).
Note that this characterization fails after dropping countability, since Mann~\cite{Man15} proved that the group of germs at infinity of orientation preserving homeomorphisms
of \(\R\) admits no nontrivial action on the real line.

An \emph{order preserving homomorphism} between left-ordered groups is a group homomorphism that preserves their orders. In particular, every order preserving homomorphism is automatically an embedding because of the anti-reflexivity and linearity of strict orders. Throughout this paper, when we talk about ordered groups, we simply say ``isomorphism'' instead of ``order preserving isomorphism'' whenever it is clear from the context.

Recall that an ordered group is \emph{Archimedean} if for every \(g,h\in G\), there is \(n\in \Z\) such that \(g < h^n\).
Archimedean left-ordered groups appeared as early as the work of Stolz~\cite{Sto} going back to 1891. Soon thereafter H\"older proved the following:

\begin{thm}[H\"older~\cite{Hol}]
\label{theorem : Holder}
Every Archimedean left-ordered group is order isomorphic to a subgroup of \((\R,+)\) equipped with the natural ordering on \(\R\).
\end{thm}

An immediate consequence is that the Archimedean type property implies commutativity. Henceforth, we shall say ``Archimedean order'' instead of ``Archimedean left-order'', and we shall use exclusively the additive notation for Archimedean (ordered) groups.
Using Theorem~\ref{theorem : Holder} one can also characterize Archimedian groups as those ordered groups whose dynamical realization is a free action on \(\R\) by orientation preserving homeomorphisms. For more details on the dynamical realization of Archimedean orders see \cite[Section~3.1]{Nav10}.
Another consequence of H\"older's theorem is the following classical lemma.

\begin{lemma}[{Hion's Lemma~\cite{Hio54}}]
\label{lem : Hion}
Suppose that \(A\) and \(B\) are two (necessarily Archimedean) subgroups of \(\R\) and \(h\colon A \to B\) is an order preserving homomorphism. Then, there exists a scalar \(\lambda \in \R^{+}\) such that \(h(a)= \lambda a\),  for every \(a\in A\). In fact, such \(\lambda\) is computed as the ratio \(\frac{h(a)}{a}\), for any \(a\in A\).
\end{lemma}

\subsection{Spaces of Archimedean orderings}
For a countable left-orderable group \(G\), let \(\LO(G)\) denote the compact Polish space of left-orderings of \(G\) in the sense of Sikora~\cite{Sik04}. 
By identifying any left-ordering \(<\) with the corresponding positive cone \(P_{<}\coloneqq\set{g\in G}{1_G<g}\) we have a one-to-one correspondence between left-orderings on \(G\) and those subsets of \(P\subseteq G\) such that
\begin{enumerate}
\item
\label{cond : lo1}
\(P \cdot P\subseteq P\);
\item
\label{cond : lo2}
\(P\sqcup P^{-1}\sqcup\{1_G\} = G\).
\end{enumerate}
In fact, if \(P\subseteq G\) satisfies \eqref{cond : lo1}--\eqref{cond : lo2} then we can define a left-compatible order \(<_P\) on \(G\) by declaring
\(g <_P h \iff g^{-1}h \in P\).
Therefore, \(\LO(G)\) can be regarded as a closed subset of \(2^G\), endowed with the product topology. Throughout this paper, we concentrate on Abelian groups, so that \(G\) is left-orderable if and only if \(G\) is torsion-free.
If \(G\) is a torsion-free group of rank \(n\) then \(\LO(G)\) and \(\LO(\Q^n)\) are Borel isomorphic. In fact, every order on \(G\) extends canonically to an order on \(\Q^n\). (We are tacitly identifying \(\Q^n\) with the divisible hull of \(G\).)

Here we focus on Archimedean orderable groups.
Let \(\Ar(G)\) be the \emph{Polish space of Archimedean orders on \(G\)}.
Notice that \(\Ar(G)\) can be construed as the \(G_\delta\) subset of \(\LO(G)\) defined by \(\Ar(G) = \{P \in \LO(G) : \forall x,y\, \exists k\in \Z \, (x^{-1}y^k \in P)\}\).

The analysis of \(\Ar(G)\), for a given Abelian group \(G\), began essentially with the work of Teh~\cite{Teh} and Minassian~\cite{Min72}.
By construction we have that whenever \(G\) is a torsion-free Abelian group of rank \(n\), then \(\Ar(G)\) is Borel isomorphic to \(\Ar(\Q^n)\).
Then we can focus only on the spaces \(\Ar(\Q^n)\). 
It is obvious that \(\Q\) can be given only two different left-orders and both of them are Archimedean: indeed,
 \(\LO(\Q) = \Ar(\Q) = \{<,>\}\).
However, Teh~\cite{Teh} showed that for \(n > 1\) there exist continuum many different Archimedean orders on \(\Q^n\), and that they can be constructed as follows.

Denote by \(e_{1} = (1,0,\dotsc,0), \dotsc, e_{n} = (0,\dotsc,0,1)\) the canonical basis of the \(\Q\)-vector space \(\Q^{n}\).
By H\"older's theorem, for any \(P\in \Ar(\Q^{n})\) there exists an order preserving group homomorphism \(\phi\colon(\Q^{n},<_P)\to (\R,<)\).

\begin{defn}
\label{def : type}
Let \(P \in \Ar(\Q^n)\) and \(\phi\) be an embedding as above.
For \(i=1,\dotsc, n\), define the \emph{type} \(\tau(P) = (\alpha_1,\dotsc, \alpha_n)\) of \(P\) by setting \(\alpha_{i}= \phi(e_{i})/{|\phi(e_{1})}|\). 
Observe that \( (\alpha_{1}, \dotsc, \alpha_{n}) \in \mathbb{R}^{n}\) is such that
\begin{enumerate}
\item
\label{it : type1}
\(\alpha_{1}= \pm 1\),
\item
\label{it : type2}
\(\alpha_{1},\dotsc,\alpha_{n}\) are linearly independent over \(\mathbb{Q}\).
\end{enumerate}
\end{defn}

Note that the type \(\tau(P)\) does not depend on the choice of \(\phi\) because of Hion's Lemma. Moreover, \(\alpha_1\) coincides with the sign of \(e_1\) under \(<_P\).

\begin{thm}[Teh~\cite{Teh}]
\label{thm : Teh}
Let \(P, S \in \Ar(\Q^{n})\) with types \(\tau(P) = (\alpha_{1},\dotsc, \alpha_{n})\) and \(\tau(S) = (\beta_{1},\dotsc, \beta_{n})\). Then, \(P=S\) if and only if \(\tau(P) = \tau(S)\).
\end{thm}

Conversely, suppose that \(\alpha = (\alpha_1, \dotsc, \alpha_n) \in \R^n\) is a type that satisfies conditions \eqref{it : type1}--\eqref{it : type2} of Definition~\ref{def : type}.
Then, the relation \(<_\alpha\) on \(\Q^n\) defined by declaring
\[
(x_1,\dotsc, x_n) <_\alpha (y_1,\dotsc, y_n) \iff \sum_i x_i\alpha_i < \sum_i y_i\alpha_i
\]
is an Archimedean order of the group \(\Q^n\) with type \(\alpha\). It follows that there is a Borel one-to-one correspondence between the elements of \(\Ar(\Q^n)\) and the standard Borel space of types for Archimedean orders on \(\Q^n\).

\begin{remark}
The above construction can also be found in \cite[Example~1.7]{ClaRol} and has an immediate geometrical interpretation. Suppose that \(\alpha = (\alpha_1, \alpha_2)\) satisfies \eqref{it : type1}--\eqref{it : type2} from above so that \(\alpha\) is a normal vector to the line \(\mathbf{r} = \set{(x,y)\in \Q^2}{ y = -\frac{\alpha_1}{\alpha_2}x}\). Then, the positive cone \(P_{<_\alpha}\) is the semiplane given by the
rational  points that lie above (or below, accordingly to the sign of \(\alpha_1\)) \(\mathbf{r}\).
For example, if \(\alpha = (1, \sqrt{2})\), then \(P_{<_\alpha} = \set{(x,y)\in \Q^2}{y\geq -\frac{1}{\sqrt{2}}x}\).
\end{remark}

Teh's work yields the existence of continuum many different Archimedean orders on \(\Q^n\), for \(n> 1\). Thus, it is very natural to ask in how many non-isomorphic way we can equip \(\Q^n\) with an Archimedean order, and whether we can classify them in a concrete way (i.e., using real numbers as complete invariants).

Let \(\cong_{\Ar(\Q^n)}\) be the isomorphism relation on the Archimedean groups with underlying set \(\Q^{n}\). Notice that \(\cong_{\Ar(\Q^n)}\) is a countable\footnote{Recall that an equivalence relation is said to be \emph{countable} if all of its classes are countable.}
Borel equivalence relation because it is induced by the action of \(\Aut(\Q^n)\), which coincides with \(\GL_n(\Q)\).
That is,
for any \(P, S \in \Ar(\Q^n)\) we have
\[
(\Q^n, <_P) \cong_{\Ar(\Q^n)} (\Q^n, <_S) \quad \iff \quad \exists \psi \in \GL_n(\Q) (\psi(P) = S).
\]

In the remainder of this section we prove Theorem~\ref{thm : Ar(Q2)}, which particularly implies that \(\cong_{\Ar(\Q^2)}\) is not concretely classifiable.
First, we briefly introduce some notions that are fundamental to the theory of countable Borel equivalence relations.

Let \(E\) and \(F\) be equivalence relations on \(X\) and \(Y\), respectively. A map \(f \colon X \to Y\) is a
\emph{homomorphism} from \(E\) to \(F\) if \(x\mathbin{E}y \implies {f(x)}\mathbin{F} {f(y)}\) for all \(x,y \in X\).
A subset \(A\subseteq X\) is said to be \(E\)-invariant if and only if it is a union of \(E\)-classes.

\begin{defn}
An equivalence relation \(E\) on a Polish space \(X\) is said to be \emph{generically ergodic} if for every Borel homomorphism from \(E\) to \(=_\R\) there is a Borel \(E\)-invariant comeager set \(C\subseteq X\) such that \(f\restriction C\) is constant.
\end{defn}

An example of generically ergodic equivalence relation is \(E_0\), which is defined as eventual equality on infinite binary sequences, that is,
 \(x \mathbin{E_0} y\) if \(\exists m \forall n\geq m (x(n) = y(n))\) for all \(x,y\in 2^\N\). Within the class of (countable) Borel equivalence relations, \(E_0\) is the immediate successor of \(=_\R\) in the sense that a Borel equivalence relation
\(E\) is not concretely classifiable if and only if \(E_0 \leq_B E\). This was proved by Harrington, Kechris, and Louveau~\cite{HarKecLou}
generalizing earlier work of Glimm~\cite{Gli61} and Effros~\cite{Eff65},
and is often called the \emph{general Glimm-Effros dichotomy}.

\begin{defn}
Suppose that \(E\) and \( F\) are countable Borel equivalence relations on the standard Borel spaces \(X\) and \(Y\), respectively. We say that \(E\) is \emph{weakly Borel reducible} to F if there exists a countable-to-one Borel homomorphism \(f \colon X \to Y\) from \(E\) to \(F\). In this case, we say that \(f\) is a \emph{weak Borel reduction} from \(E\) to \(F\).
\end{defn} 

A consequence of the fact that \(E_0\) is generically ergodic is that there is no weak Borel reduction from \(E_0\) to \(=_\R\).
This fact will be crucially used to prove Theorem~\ref{thm : Ar(Q2)}.

\begin{proof}[Proof of Theorem~\ref{thm : Ar(Q2)}]
Consider the group \(\GL_2(\Z)\) of all matrices
\(\begin{psmallmatrix} a & b\\ c & d \end{psmallmatrix}\) with integer coefficients and
determinant \( ad-bc = \pm 1 \).
We let \(\GL_2(\Z)\) act on \(\R\cup\{\infty\}\) by fractional linear transformations, that is, \(\begin{psmallmatrix} a & b\\ c & d \end{psmallmatrix}\cdot x = \frac{ax+b}{cx+d}\). (We stipulate that \(\frac{ax+b}{cx+d} = \infty\) when \(cx+d=0\) or \(x=\infty\).)

When \(X\subset \R\cup\{\infty\}\) is \(\GL_2(\Z)\)-invariant we denote by \(E^{X}_\GL\) the equivalence relation on \(X\) induced by \(\GL_2(\Z)\).
Since all elements of \(\Q\cup\{\infty\}\) form a unique class,
it is immediate that \(E^{\R\smallsetminus \Q}_\GL\) is Borel bireducible with \(E^{\R\cup \{\infty\}}_\GL\).
Further, it is well-known that \(E^{\R\cup\{\infty\}}_\GL\) is Borel bi-reducible with \(E_0\). (E.g., see~\cite[Example~1.4(C)]{JacKecLou}.)

\begin{proof}[Claim]
There is a weak Borel reduction from \(E^{\R\smallsetminus \Q}_\GL\) to \(\cong_{\Ar(\Q^2)}\).
\renewcommand{\qed}{}
\end{proof}

The above claim implies that there is a weak Borel homomorphism from \(E_0\) to \(\cong_{\Ar(\Q^2)}\). Now, if there was a Borel reduction from \(\cong_{\Ar(\Q^2)}\) to \(=_\R\), then one could define a weak Borel isomorphims from \(E_0\) to \(=_\R\) by composing such reduction with the previous homomorphism, a contradiction.
It follows that \({\cong_{\Ar(\Q^2)}} \nleq_B {=_\R} \), and thus \(E_0 \leq_B {\cong_{\Ar(\Q^2)}}\) by the general Glimm-Effros dichotomy.

It remains to prove the claim. 
For each \(\alpha \in {\R\smallsetminus \Q}\), consider the group defined as
\(\spann_\Q \{1, \alpha\}\) with the obvious ordering inherited from \(\R\).
If \(\alpha, \beta \in {\R\smallsetminus \Q}\) are such that
\(\begin{psmallmatrix} a & b\\ c & d \end{psmallmatrix}\cdot \alpha = \beta\), for some \(\begin{psmallmatrix} a & b\\ c & d \end{psmallmatrix} \in \GL_2(\Z)\),
then the group \(\spann_\Q \{1, \alpha\}\) is order isomorphic to \(\spann_\Q \{1, \beta\}\).
To see this, notice that \(|a\alpha+b|\) and \(|c\alpha+d|\) are linearly independent over \(\Q\). Let \(\lambda \coloneqq 1 / |c\alpha + d|\). Since \(\lambda |c\alpha+d| = 1 \) and \(\lambda |a\alpha+b| = |\beta|\),  the map given by the scalar multiplication \(x\mapsto \lambda x\) sends a basis of \(\spann_\Q \{1, \alpha\}\) into a basis of \(\spann_\Q \{1, \beta\}\).
It follows that \(x\mapsto \lambda x\) is an order preserving isomorphism as desired.

Finally we show that given \(\beta \in \R\smallsetminus \Q\), there are at most countably many \(\alpha \in \R\smallsetminus \Q\) so that
\(\spann_\Q \{1, \alpha\}\) is order isomorphic to \(\spann_\Q \{1, \beta\}\).
It will follow that the map \(\alpha\mapsto \spann_\Q \{1, \alpha\}\) is
countable-to-one, and hence a weak Borel reduction from \(E_\GL^{\R\smallsetminus \Q}\) to \(\cong_{\Ar(\Q^2)}\).
Suppose that \(\spann_\Q \{1, \alpha\}\) is order isomorphic to \(\spann_\Q \{1, \beta\}\). By Hion's Lemma there exists some \(\lambda\in \R^+\) such that the map \(x\mapsto\lambda x \) realizes an isomorphism from \(\spann_\Q \{1, \alpha\}\) to \(\spann_\Q \{1, \beta\}\). It follows that \(\lambda = \lambda\cdot 1 \in \spann_\Q \{1, \beta\}\) and thus \(\lambda = m\beta+ n\)  for some \(m,n \in \Q\). Similarly, we obtain \(\lambda \alpha \in \spann_\Q \{1, \beta\}\), which implies that \((m\beta+n)\alpha = k\beta +\ell\) for some \(k,\ell \in \Q\). It follows that \(\alpha = \frac{k\beta+\ell}{m\beta+n}\) for some \(k,\ell,m,n \in \Q\), in particular there are only countably many possible values for \(\alpha\).
\end{proof}

As recalled at the beginning of this section, it has been well-known since the work of Teh~\cite{Teh} that \(\Ar(\Q^2)\) is uncountable.
Theorem~\ref{thm : Ar(Q2)} implies the stronger fact that
 \(\Q^2\) admits continuum many Archimedean orders \emph{up to order preserving isomorphism}.
In fact, it shows that the \emph{Borel cardinality} of \(\Ar(\Q^2)\) modulo the isomorphism relation is strictly bigger than the one of \(\R\) modulo identity.
That is, there is no injection from \(\Ar(\Q^2)/{\cong_{\Ar{\Q^2}}}\) into \(\R\) admitting a Borel lifting.

Another way of rephrasing Theorem~\ref{thm : Ar(Q2)} is saying that the quotient Borel space \(\Ar(\Q^2)/\GL_2(\Q)\) is not standard. The quotient structures of spaces of left-orderings has been already investigated by Calderoni and Clay~\cite{CC22}, who proved that the quotient Borel space  \(\LO(G)/\Inn(G)\) is not standard for a large class of groups including all nonabelian groups that are not locally indicable and the nonabelian free groups on \(n\) generators, for all \(n\geq 2\).

\begin{question}
Is \(\cong_{\Ar(\Q^2)}\) Borel reducible to \(E_0\)?
\end{question}

\begin{question}
What is the Borel complexity of \(\cong_{\Ar(\Q^n)}\) for \(n>1\)?
\end{question}

\subsection{Archimedean groups of rank 1}

In this short subsection we show that also Baer's analysis of torsion-free Abelian groups of rank \(1\) shows that we cannot use numerical invariant to classify completely \(\cong_\mathsf{ArGp}\).

It is well-known that an Abelian group is torsion-free of rank \(1\) if and only if it is a subgroup of \(\Q\), and the same remains true if we move to rank \(1\) ordered Abelian groups and equip \(\Q\) with its natural order. 
Let \(\mathsf{S}(\Q)\) be the set of subgroups of \(\Q\). It is easily checked that \(\mathsf{S}(\Q)\) is a \(G_{\delta}\) subset of \(2^{\Q}\), thus a Polish space. All elements of \(\mathsf{S}(\Q)\) will be regarded as ordered groups with the obvious ordering. Let \(\cong_{\mathsf{S}(\Q)}\) be the isomorphism relation on \(\mathsf{S}(\Q)\), and let \(\cong^o_{\mathsf{S}(\Q)}\) be the ordered isomorphism relation on \(\mathsf{S}(\Q)\).

The following well-known statement rephrases Baer's classification theory~\cite{Bae} in terms of Borel reducibility. 

\begin{prop}[See {\cite[Theorem~1.3]{Tho02}}]
\label{prop : Baer}
\({E_{0}} \sim_{B} {\cong_\mathsf{S(\Q)}}\).
\end{prop}

Moreover, one can show directly that \(\cong_{\mathsf{S}(\Q)}\) and \(\cong^o_{\mathsf{S}(\Q)}\) coincide. First, we isolate the following fact: 

\begin{lemma}
\label{lem : rank 1}
Let \(A,B \in {\mathsf{S}(\Q)}\). If \(f\colon A \to B\) is a group homomorphism, then there exists a nonzero \(q\in \Q\) such that \(f(a) = q a\), for all \(a \in A\).
\end{lemma}
\begin{proof}
It suffices to show that \(f(a)/a = f(b)/b\) for all nonzero \(a,b\in A\).
Since \(A\) is rank \(1\), \(a\) and \(b\) are linearly dependent, that is, there are nonzero \(m,n \in \Z\) such that \(ma+nb = 0\).
Write \(b = -\frac{m}{n} a\). Then
\[
a f(b) = \frac{n}{n} a f(-\frac{m}{n} a) = \frac{1}{n} a f(-ma) = -\frac{m}{n} a f(a) = b f(a),
\]
where the second and third equalities follow from the fact that \(f\) is a homomorphism. Then the result follows.
\end{proof}

\begin{prop}
The relations \(\cong_{\mathsf{S}(\Q)}\) and \(\cong^o_{\mathsf{S}(\Q)}\) coincide, that is:
\(A\cong_{\mathsf{S}(\Q)} B\) if and only if \(A\cong^o_{\mathsf{S}(\Q)} B\), for all \(A,B\in \mathsf{S}(\Q)\).
\end{prop}

\begin{proof}
The proposition follows from Lemma~\ref{lem : rank 1}. If \(A\cong_{\mathsf{S}(\Q)} B\) and \(f\colon A \to B\) is an isomorphism, then there is a nonzero \(q\in \Q\) such that \(f(a) = q a\) for all \(a \in A\). If \(q> 0\), then \(f\) is clearly order preserving. Otherwise \( f \) is order-reversing, and hence the map \(a\mapsto -f(a)\) is an order preserving isomorphism between \(A\) and \(B\).
\end{proof}

\section{Bounds for \(\cong_\mathsf{ArGp}\)} \label{sec:boundsforisomorphism}

Let \(\mathcal{L} = \{ {+}, {<} \}\) be the language for ordered groups.
We can define the Polish space \(X_\mathcal{L}\) of all \(\mathcal{L}\)-structures with domain \(\N\) in the usual fashion. (E.g., see~\cite[Chapter~2.3]{Hjo00}.)
Next, we define \(X_\mathsf{ArGp}\) as the subset of \(X_\mathcal{L}\) consisting of those structures satisfying the axioms of ordered Archimedean groups. One sees that \(X_\mathsf{ArGp}\) is \(G_\delta\), thus it is a Polish space with the topology induced by \(X_\mathcal{L}\). Also, let \(\cong_\mathsf{ArGp}\) 
be 
the isomorphism relation
on \(X_\mathsf{ArGp}\).

In view of H\"older's theorem (cf. Theorem~\ref{theorem : Holder}), it will be convenient to work also with a different coding space. For any set \(X\), denote by \(X^{(\N)}\) the set of injective sequences in \(X\). Then let
\[
\mathcal {A} \coloneqq \set{(x_{n})_{n\in \N}\in \R^{(\N)}}{x_{0}=0 \text{ and } (\{x_{n}\}_{n\in\N},+)< \R},
\] 
which is a closed subset of \(\R^{\N}\), thus a Polish space with the topology inherited from it.
With a slight abuse of notation, we make no distinction between sequences \(\vec{x} = (x_n)_{n\in \N}\) in \(\mathcal{A}\) and the groups enumerated by them.
We denote by \(\cong_{\mathcal{A}}\) 
the isomorphism relation
on \(\mathcal{A}\).

Notice that
the two ways of coding countable Archimedean groups into a Polish space are equivalent. If \((x_{n})_{n\in \N}\) is in \(\mathcal{A}\), then we define a corresponding ordered-group structure \((\N, +^G, <^G)\) by setting \( n +^G m = k \iff x_{n} + x_{m} = x_{k}\) and \(n <^G m \iff x_{n} < x_{m}\), for every \(n,m\in \N\).
Conversely, we have the following proposition which together with the previous observation yields
\[ 
{\cong_\mathsf{ArGp}} \sim_B {\cong_\mathcal{A}}. 
 \]

\begin{prop}
\label{prop : continuous}
There is a continuous map \(X_\mathsf{ArGp} \to \mathcal{A}\) sending any \(G \in X_\mathsf{ArGp}\) to (an enumeration of) a group \(\vec{x}_G = (x_G(i))_{i\in \N}\) in \(\mathcal{A}\) isomorphic to \(G\).
In particular, for \(G,H \in X_\mathsf{ArGp}\) we have that \(G\) is isomorphic to (respectively, embeds into) \(H\) if and only if \( \vec{x}_G \) is isomorphic
to (respectively, embeds into) \( \vec{x}_H \).
\end{prop}

\begin{proof}
Every element of \(X_\mathsf{ArGp}\) is a structure of the form \(G = (\N, +^G, <^G)\) where \(+^G\) is a commutative group operation on the set of natural numbers, and \(<^G\) is an Archimedean order. To simplify our notation we shall assume that \(0\) is the neutral element without losing generality. If \(n\in\N\) and \(a\in G\) we shall use the notation \(na = a+^G\dotsb+^Ga\). (It is important to make distinction between \(+^G\) and \(<^G\) with the natural addition and the natural ordering on \(\N\).)

 First, set \(x_G(0) = 0\).
Next, let \(\ell\) be the least natural number such that \(0<^G\ell\).  For each \(0\neq t \in \N\), if \(0<^G t\), then define
\begin{align*}
    L^G_t &= \set{\frac{m}{n} \in \Q}{ {m \ell} <^G {nt}},\\
    R^G_t &= \set{\frac{m}{n} \in \Q}{ {m  \ell} \mathbin{\prescript{G}{}{\geq}} {n t}}. 
\end{align*}

\begin{proof}[Claim]
\((L^G_t, R^G_t)\) is a Dedekind cut in \(\Q^+\).
\renewcommand{\qed}{}
\end{proof}

\begin{proof}[Proof of the claim]
Since \(G\) is Archimedean it is clear that both \(L^G_t\) and \(R^G_t\) are nonempty, they are disjoint, and \(L^G_t\cup R^G_t = \Q^+\), since \(<^G\) is a total order.
To see that \(L^G_t\) is downward closed, let \(\frac{m}{n}\in L^G_t\) and let \(\frac{m'}{n'} < \frac{m}{n}\).
Since \(m\ell <^G nt\), we have \(mn'\ell <^G nn't\). By the assumption on \(\frac{m'}{n'}\), it follows that \(m'n\ell <^G nn't\). Thus, we have \(m'\ell <^G n't\), which implies that \(\frac{m'}{n'}\in L^G_t\), as desired.
\end{proof}
 
Set \(x_G(t) \coloneqq \sup L_t = \inf R_t\). If \(t<^G 0\), set \(x_G(t) = - x_G(-^Gt)\).
It is not hard to check that the map \(i\mapsto x_G(i)\) is an order preserving group isomorphism from \(G\) to \(\vec{x}_G = (x_G(i))_{i\in \N}\), regarded as a subgroup of \(\R\). 

It remains to show that the map \(G\mapsto \vec{x}_G\) is continuous. Since \(\mathcal{A}\) is given the product topology, it suffices to show that for all
nonzero \(t\in \N\) and \(r\in \R\) the sets
\[
U^t_{>r} = \set{G\in X_\mathsf{ArGp}}{ x_G(t) > r}\quad\text{and}\quad U^t_{< r} =  \set{G\in X_\mathsf{ArGp}}{ x_G(t) < r}
\]
are open.  Recall that a basis for the topology on \(X_\mathsf{ArGp}\) is given by the sets
\[
V_{H,F} = \set{G\in X_\mathsf{ArGp}}{ G \restriction F = H\restriction F}
\]
with \(H \in X_\mathsf{ArGp}\) and \(F\) a finite subset of \(\N\).

The case \( r = 0 \) is clear, as \( U^t_{>r} \) (respectively, \( U^t_{< r} \)) consists of those \( G \in X_\mathsf{ArGp} \) for which \( t >^G 0 \) (respectively, \( t <^G 0 \)). 

Let us now assume \( r > 0 \) and consider the case of \( U^t_{>r} \).
Given any \(H \in X_\mathsf{ArGp}\) with \(x_H(t) > r\), we need to find a finite set \( F \subseteq \N \)  such that \(x_G(t) > r\) for all \(G\in V_{H,F} \).
Since \( r \) is positive, for all \(G \in X_\mathsf{ArGp}\) with \( 0 <^G t \) we have

\begin{equation}
\label{eq : A_G(t)}
x_G(t) > r\quad \iff\quad  
\text{there is $\frac{m}{n} \in L_t^G$ such that $r<\frac{m}{n}$.}
\end{equation}
In particular, equation~\eqref{eq : A_G(t)} applies to \( H \), because  \( x_H(t) > r > 0 \) implies \( 0 <^H t \). Fix \(\frac{m_H}{n_H}\) as in \eqref{eq : A_G(t)} witnessing that \(x_H(t) > r\), and let \( \ell \) be the natural number used to compute \( x_H(t) \), namely the least natural number such that \( 0 <^H \ell \). By \( \frac{m_H}{n_H} \in L^H_t \) we have \( m_H \ell \leq^H n_H t \).
Define
\[t_i = \underbrace{t +^H \dotsb +^H t}_{i \text{ times}}\quad \text{and}\quad
s_j = \underbrace{\ell +^H \dotsb +^H \ell}_{j \text{ times}}
\] for \(i, j \in \N\), and let \(F\coloneqq \set{k\in \N}{k\leq \ell}\cup\set{t_i}{1\leq i\leq n_H}\cup \set{s_j}{ 1 \leq j\leq m_H}\).

We claim that such \( F \) works, i.e.\ that \(x_G(t)> r\) whenever  \(G \in V_{H,F}\).
Indeed, since \(G\restriction \set{k\in \N}{k\leq \ell}\) and \(H\restriction \set{k\in \N}{k\leq \ell}\), the natural number \(\ell\) is also the least such that \(0<^G\ell\). Moreover, for all \(1\leq i < n_H\) and \(1\leq j < m_H\) we have \(t_{i+1} = t +^H t_{i}\) and
\(s_{j+1} = \ell +^H s_j\). By definition all those sums are in \(F\), therefore
\[
t_{i+1} = {t_i} +^H t = {t_i} +^G t\quad \text{ and }\quad s_{j+1} = {s_j} +^H \ell = s_{j} +^G \ell
\]
because \( G \restriction F = H \restriction F \).
Also, since the order relations \(<^G\) and \(<^H\) agree on \(F\) we have
\[
m_H \ell \leq^G n_H t.
\]
It follows that \( \frac{m_H}{n_H} \in L^G_t\), and since \( r < \frac{m_H}{n_H} \) we get that \( \frac{m_H}{n_H} \) witnesses \(x_G(t) > r\), as desired.

As for the case of \( U^t_{<r} \) with \( r > 0 \), we fix \( H \in X_\mathsf{ArGp} \) with \( x_H(t) < r \) and distinguish two cases. If \( t <^H 0 \), then the neighborhood of \( H \) consisting of those \( G \in X_\mathsf{ArGp} \) with \( t <^G 0 \) is contained in \( U^t_{<r} \) and we are done. Otherwise, we argue as for \( U^t_{> r} \) but replacing \( L^G_t \) with \( R^G_t \) and reversing all inequalities in the appropriate way.

Finally, the case of a negative \( r \) is reduced the positive case by observing that \( U^t_{> r} \) is the union over all \( s \in \N \) of the open sets 
\[
\set{G \in X_\mathsf{ArGp}}{s = -^G t} \cap U^s_{< -r},
\] 
and analogously for \( U^t_{< r} \).
\end{proof}

Notice that Hion's Lemma implies that \(\cong_\mathcal{A}\) (and hence also \( \cong_\mathsf{ArGp} \))
is a Borel equivalence relation.
In fact, if \(A\) and \(B\) are countable ordered subgroups of \(\R\) and \(h\colon A \to B\) is an ordered preserving morphism, then \(f\) is a scalar multiplication by some \( \lambda \in \R^+\). More precisely, for any given \(a\in A\), we have \( \lambda = h(a)/a\) and this gives countably many possible choices for \( \lambda \).
Therefore, between any two countable subgroups of the reals there are at most countably many order preserving homomorphisms. Then \(\cong_\mathcal{A}\) is the projection of a Borel set with countable sections, and is thus Borel itself.

In the remainder of this section, using the theory of Hjorth, Kechris, and Louveau~\cite{HjoKecLou}
we leverage the descriptive theoretical analysis of \(\cong_\mathsf{ArGp}\) to find an upper bound for its complexity within the class of Borel equivalence relations.
First let us recall a definition of Friedman and Stanley~\cite{FriSta}.

If \(E\) is an equivalence relation on a standard Borel space \(X\) we can define the \emph{Friedman-Stanley jump} \(E^{+}\) on \(X^\N\), the space of countable sequences in \(X\), by declaring 
\[
(x_i)_{i\in \N} \mathrel{E}^+ (y_i)_{i\in \N} \iff \set{[x_i]_E}{i\in \N} = \set{[y_i]_E}{i\in \N}.
\]

For example, for all real valued sequences \((x_i)_{i\in \N}, (y_i)_{i\in\N}\), we have \mbox{\( (x_i)_{i\in\N} =^+_\R (y_i)_{i\in\N}\)} precisely when they enumerate the same subset of \(\R\). Therefore, countable sets of reals up to equality can be used as complete invariants for \(=^+_\R\).

More generally, one can recursively define for any countable ordinal \(\alpha\) the \( \alpha \)-th iterated jump \( =_\R^{\alpha+} \) of \( =_\R \) letting \(=_\R^{(\alpha+1)+}\) be \((=_\R^{\alpha+})^{+}\) and, 
for \( \gamma <\omega_1\) limit, defining \(=_\R^{\gamma +}\) as the product \(\prod_{\alpha<\gamma} {=_\R^{\alpha+}}\).
Notice that hereditarily countable subsets in \(\mathcal{P}^{1+\alpha}(\N)\) serve as complete invariants for \(=_\R^{\alpha+}\).

We call \emph{isomorphism relations} those equivalence relations that are defined as the model-theoretic notion of isomorphism \( \cong \) on the standard Borel space of countable models of an \(\mathcal{L}_{\omega_1\omega}\)-sentence in a given countable language. (Here \(\mathcal{L}_{\omega_1\omega}\) is the
infinitary version of first-order logic that allows formulas with countable conjunctions and disjunctions.) Equivalently, an isomorphism relation is the restriction of \( \cong \) to
some \( \cong \)-invariant Borel subset of \( X_\mathcal{L} \) for some countable first-order language \( \mathcal{L} \).
The hierarchy \(\set{=_\R^{\alpha+}}{\alpha<\omega_1}\) stratifies the class of Borel isomorphism relations in the following precise sense:

\begin{fact}
For all \(\alpha <\omega_1\), the following hold:
\begin{itemize}
    \item 
    \({=_\R^{\alpha+}} <_B {=_\R^{(\alpha+1)+}}\),
    \item
    There is a Borel isomorphism relation          \(\cong_{1+\alpha} \) such that \(=_\R^{\alpha+}\) is  Borel bi-reducible%
    \footnote{The gap in the indexing is due to the fact that when introducing the notation in terms of isomorphism relations, the authors of~\cite{HjoKecLou} considered the Friedman-Stanley hierarchy as starting from equality \( =_\N \) on natural numbers, so that \( {=^+_\N} \sim_B {=_\R} \), \( {=^{++}_\N } \sim_B {=_\R^+} \), and so on. Since for ease of notation we are instead starting from \( =_\R \) as the base of the hierarchy, we have a shift of one unit in the indexes, so that \( {=_\R^+} \sim_B {\cong_2} \), \( {=_\R^{++}} \sim_B {\cong_3} \), and so on.} 
    with  \(\cong_{1+\alpha}\).
\end{itemize}
Moreover, every Borel isomorphism relation is Borel reducible to \(\cong_\alpha\), for some \(\alpha \in \omega_1\).
\end{fact}

\subsection{A lower bound for \(\cong_\mathsf{ArGp}\)}
Below we discuss why the relation of equality \(=^+_\R\) on countable subsets of reals is a natural lower bound for \(\cong_\mathsf{ArGp}\).

\begin{prop}
\label{prop : fields}
Suppose that \(A, B\) are subfields of \(\R\). Then, the following conditions are equivalent:
\begin{enumerate}
\item \label{itm:1}
\(A\subseteq B\);
\item \label{itm:2}
There is an order preserving (group) homomorphism \(h\colon A\to B\).
\end{enumerate}
Thus, \(A\) and \(B\) are bi-embeddable as ordered groups if and only if they are order isomorphic if and only if \(A = B\).
\end{prop}
\begin{proof}
The nontrivial implication is from \eqref{itm:2} to \eqref{itm:1}. Assume that \(A,B\) are subfields of \(\R\) so that in particular \(1\in A,B\) and \(A,B\) are closed under inverse. 
By Hion's Lemma there is \(\lambda \in \R^+\) such that \(f(a) = \lambda a\) for all \(a\in A\). Clearly, when \(a = 1\) we obtain that \(f(1) = \lambda \in B\).
Therefore, for any \(a\in A\), we have \(a = f(a)\lambda^{-1}\), which is an element of \(B\) because \(B\) is closed under inverses and multiplication.
\end{proof}

For \(S \subseteq \R\), we denote by \(\Q(S)\) the subfield of \(\R\) generated by \(S\).
As usual, whenever \(S=\{x_{0},\dotsc,x_{n}\}\) is finite,  we write \({\displaystyle \Q(x_{1},\ldots ,x_{n})}\) instead of \( \Q(\{x_{1},\ldots ,x_{n}\})\). 

A set \(T\subseteq \R \) is called \emph{algebraically independent (over \(\mathbb{Q}\))} 
if \(P(t_1 , \dotsc, t_n )\neq \)0 for any \(t_1 , \dotsc, t_n \in T\) and nontrivial polynomial \(P\) with rational coefficients.
In particular, if \(T\) is algebraically independent each \(t \in T\) is transcendental over
\(\Q (T \smallsetminus \{t\})\).
In this case the field \(\Q(T)\) consists of all fractions of polynomials in finitely many variables taken from \(T\).

Recall that a subset of a topological space is \emph{perfect} if it is closed and has no isolated points (in its relative topology).  In particular, every perfect subset of a Polish space can be regarded as a Polish space with its relative topology. For our next results, we will use that every nonempty perfect Polish space has cardinality the continuum (e.g., see \cite[Corollary~6.3]{Kec}), thus is Borel isomorphic to \(\R\).
We point out the following consequence of a result of Mycielski~\cite{Myc64}, see also \cite[Theorem~7.5]{Wag}.

\begin{fact}
\label{fact : perfect}
There is a perfect set \(Y\subseteq \R\) such that for every finite sequence \(a_{0},\dotsc, a_{n}\in Y\), for every \(0\leq i\leq n\), the real
\(a_{i}\) is transcendental over \(\Q(a_{0},\dotsc a_{i-1},a_{i+1},\dotsc a_{n})\).
\end{fact}

\begin{prop}
\label{prop : lowbnd}
\({=^+_\R} \leq_B {\cong_\mathcal{A}}\).
\end{prop}

\begin{proof}
Let \(T\subseteq \R\) be a perfect set as in Fact~\ref{fact : perfect} and let \(f\colon \R \to T\) be a Borel isomorphism. 
Then the map sending \((x_n)_{n\in \N} \in \R^\N \) into (an enumeration of) \( \Q(\set{f(x_n)}{n \in \N})\) sends sequences enumerating different sets to different fields, which are not isomorphic as ordered groups by Proposition~\ref{prop : fields}.
\end{proof}

\subsection{An upper bound for \(\cong_\mathsf{ArGp}\)}

First, we recall the following definition of Hjorth, Kechris, and Louveau~\cite[Section~1]{HjoKecLou}.

\begin{defn}[Hjorth-Kechris-Louveau \cite{HjoKecLou}]
Fix $n\geq 3$ and $0\leq k\leq n-2$.
Let $\mathcal{P}_\ast^{n,k}(\mathbb{N})$ be the collection of pairs $(A,R)$ such that:
\begin{enumerate}
    \item $A$ is a hereditarily countable set in $\mathcal{P}^n(\N)$;
    \item $R$ is a ternary relation on $A\times A\times(\mathcal{P}^k(\N)\cap \mathrm{tc}(A))$ such that
    \begin{itemize}
    \item for any $a,b\in A$ there is some $r$ such that $R(a,b,r)$ holds;
    \item  given any $a\in A$, for any $b,b'\in A$ and any $r$, if $R(a,b,r)$ and $R(a,b',r)$ both hold then $b=b'$.
    \end{itemize}
\end{enumerate}
The equivalence relation $\cong^\ast_{n,k}$ is defined as the isomorphism relation of countable structures coding pairs $(A,R)$ in $\mathcal{P}^{n,k}_\ast(\mathbb{N})$. In other words, $\cong^\ast_{n,k}$ is defined precisely so that it admits a natural classification with $\mathcal{P}_\ast^{n,k}(\mathbb{N})$ as a set of complete invariants. See \cite[p. 95, 98, 99]{HjoKecLou} for a precise presentation of $\cong^\ast_{n,k}$ on a Polish space. 
\end{defn}

Given an invariant $(A,R)$ for $\cong^\ast_{n,k}$, the set $A$ is an invariant for $\cong_n$, while $R$ provides a parametrization of $A$ using lower rank sets, therefore ``simplifying its complexity''. Thus the invariants in $\mathcal{P}^{n,k}_\ast(\mathbb{N})$ should be viewed as sets with intermediate complexity, between $\mathcal{P}^{n-1}(\mathbb{N})$ and $\mathcal{P}^n(\mathbb{N})$. Indeed, the equivalence relations of Hjorth-Kechris-Louveau refine the Friedman-Stanley hierarchy: 
\[
{=^{(n-1)+}_\R} \sim_B {\cong_{n}}\leq_B {\cong^\ast_{n+1,0}}\leq _B \dotsb \leq_B {\cong^\ast_{n+1,n-1}}\leq_B{\cong_{n+1}} \sim_B {=^{n+}_\R}.
\]

\begin{remark}\label{rmk : invs for 3,0}
Suppose $(A,R)$ is in $\mathcal{P}_\ast^{n,0}(\mathbb{N})$. Given $a\in A$, the relation $R$ allows to enumerate $A$, sending $b\in A$ to the smallest $j \in\mathbb{N}$ such that $R(a,b,j)$ holds.
Thus invariants for $\cong^\ast_{n,0}$ should be viewed as sets in $\mathcal{P}^n(\mathbb{N})$, which can be enumerated definably in a parameter.
\end{remark}

\begin{remark}
Suppose $(A,R)$ is in $\mathcal{P}^{n,k}_\ast$ and assume additionally that for any $a,b\in A$ there is a unique $r$ satisfying $R(a,b,r)$. Then we can view this invariant as a set $A$ in $\mathcal{P}^n(\mathbb{N})$ together with injective maps $R(a,\cdot ,\cdot)$ into the lower rank sets from $\mathcal{P}^k(\mathbb{N})$, definable uniformly in a parameter $a\in A$.
In the general case, instead, we just get injective maps $b\mapsto\set{r}{R(a,b,r)}$ from $A$ to $\mathcal{P}^{k+1}(\mathbb{N})$, with the additional property that any two distinct sets in the image are disjoint.
\end{remark}

\begin{prop}[Hjorth-Kechris-Louveau~\cite{HjoKecLou}]
\({{}=_\R^{+}} <_B {\cong^{*}_{3,0}} <_B {\cong^{*}_{3,1}} <_B {=^{++}_\R}\).
\end{prop}
See also \cite[Section 6.2]{Sha} for a proof that ${\cong^\ast_{3,1}}\not\leq_B{\cong^\ast_{3,0}}$ using the method we employ in Section~\ref{sec : irred to 3,0}.
Reference \cite{Sha} further establishes that the Hjorth-Kechris-Louveau hierarchy is strict. For example, ${\cong^\ast_{n,k}}<_B{\cong^\ast_{n,k+1}}$, whenever defined.

Next, we briefly recall the theory of potential complexity.

\begin{defn}[Louveau~\cite{Lou94}] \label{def : potential classes}
Let \(\boldsymbol{\Gamma}\) be a class of sets in Polish spaces that is closed under continuous preimages.
Suppose that \(E\) is an equivalence relation on a standard Borel space \(X\). We say that \(E\) is \emph{potentially in \(\boldsymbol{\Gamma}\)} if there exists a Polish topology \(\tau\) on \( X \) generating its Borel structure  such that \(E\) is in \(\boldsymbol{\Gamma}\) in the product space \((X\times X, \tau^{2})\).
\end{defn}

By a standard change of topology argument (see e.g.~\cite[Section~13]{Kec}), \(E\) is potentially in \(\boldsymbol{\Gamma}\) if and only if there is a Polish space \(Y\) and an equivalence relation \(F\) on \(Y\) such that \(F\) is in \(\boldsymbol{\Gamma}\) (as a subset of \(Y^2\)) and \(E \leq_B F\).

Hjorth, Kechris, and Louveau \cite{HjoKecLou} completely classified the possible pointclasses $\boldsymbol{\Gamma}$ which can be realized as the optimal potential complexity of an isomorphism relation.
More precisely, we say that $\boldsymbol{\Gamma}$ is \textit{the potential complexity} of $E$ if $E$ is potentially $\boldsymbol{\Gamma}$ but $E$ is not potentially $\check{\boldsymbol{\Gamma}}$, where $\check{\boldsymbol{\Gamma}}$ is the class of complements of sets in $\boldsymbol{\Gamma}$.
Then the potential complexities of $=^+_\R$ and $=^{++}_\R$ are $\mathbf{\Pi}^0_3$ and $\mathbf{\Pi}^0_4$ respectively, while both $\cong^\ast_{3,0}$ and $\cong^\ast_{3,1}$ have potential complexity $\mathrm{D}(\mathbf{\Pi}^0_3)$, where $\mathrm{D}(\mathbf{\Pi}^0_3)$ is the class of differences of \(\mathbf{\Pi}^0_3\) sets.

The following proposition summarizes some facts from~\cite[Theorem~4.1 and Corollary~6.4]{HjoKecLou} which will be relevant here.%
\footnote{Reference~\cite{HjoKecLou} contains similar results for $\cong_n$ and $\cong^\ast_{n,n-2}$ and further generalizations of these results through the countable ordinals.}

\begin{prop}
\label{prop : HjoKecLou}
Let \(E\) be an isomorphism relation.
Then
\begin{itemize}
    \item
    \(E\) is potentially \(\boldsymbol{\Pi}^0_3\) if and only if
    \(E\leq_B{=^+_\R}\).
    \item
    \(E\) is potentially \(\boldsymbol{\Sigma}^0_4\)  if and only if \(E \leq_B{\cong^*_{3,1}}\) if and only if $E$ is potentially~$\mathrm{D}(\boldsymbol{\Pi}^0_3)$.
\end{itemize}
Moreover, no other potential complexity properly occurs in between \( \boldsymbol{\Pi}^0_3 \) and \( \boldsymbol{\Pi}^0_4 \): if an isomorphism relation is potentially \( \boldsymbol{\Sigma}^0_4 \) but not potentially \( \boldsymbol{\Pi}^0_3 \), then its  potential class is $\mathrm{D}(\boldsymbol{\Pi}^0_3)$.
\end{prop}

\begin{prop}
\label{prop : reduction to cong_3,1}
The isomorphism relation \(\cong_\mathsf{ArGp}\) is \(\boldsymbol{\Sigma}^0_4\), and thus \({\cong_\mathsf{ArGp}}\leq_B {\cong^{*}_{3,1}}\).
\end{prop}

\begin{proof}
In view of Hion's lemma~\ref{lem : Hion}, whenever \(\phi\colon A\to B\) is an order preserving homomorphism, then \(\phi\) is the scalar multiplication by \(\phi(a)/a\), for any \(a\in A\).  So for every \((a_{n})_{n\in \N}\) and \((b_{n})_{n\in \N}\) in \(\mathcal{A}\),
\begin{multline}
\label{eq:complexity}
\tag{\(*\)}
(a_{n})_{n\in \N}\cong_\mathcal{A} (b_{n})_{n\in \N}\quad \iff\quad  \exists k\in \N\; \bigg( \frac{b_{k}}{a_{1}}>0\\
\text{ and }
 \forall n\in\N\, \exists \ell\in \N \bigg(b_{\ell} = \frac{b_{k}}{a_{1}} a_n\bigg)
\text{ and } 
 \forall \ell\in\N \, \exists n\in \N \bigg(b_{\ell} = \frac{b_{k}}{a_{1}} a_n
 \bigg)
\bigg).
\end{multline}
By equation~\eqref{eq:complexity} it is clear that \(\cong_\mathcal{A}\) is \(\boldsymbol{\Sigma}^{0}_{4}\), which yields that \(\cong_\mathsf{ArGp}\) is
\(\boldsymbol{\Sigma}^{0}_{4}\) by Proposition~\ref{prop : continuous}.
\end{proof}

It follows from Proposition~\ref{prop : reduction to cong_3,1}, Theorem~\ref{thm : main} (which will be proved in the next section), and the preceding discussion, that the potential complexity of $\cong_\mathsf{ArGp}$ is precisely $\mathrm{D}(\mathbf{\Pi}^0_3)$. Furthermore, it is ``relatively complex'', for a $\mathrm{D}(\mathbf{\Pi}^0_3)$ relation, in the sense that it is not Borel reducible to $\cong^\ast_{3,0}$.

We conclude this discussion pointing out that a particular case of our analysis was already observed by Kechris~\cite[Section~8]{Kec99}. Since to the best of our knowledge no explicit proof of such result appeared in the literature, we include it here for the sake of completeness.
Let \(X_\mathsf{ArGp}^{*}\) be  the standard Borel space of countable Archimedean groups with a distinguished positive element, and let \(\cong^*_\mathsf{ArGp}\) be the corresponding isomorphism relation.

\begin{prop}[Kechris]
 \({\cong_\mathsf{ArGp}^{*}} \sim_{B} {=_\R^{+}}\).
\end{prop}

\begin{proof}
It can be shown that \({=_\R^{+} } \leq_B {\cong_\mathsf{ArGp}^{*}}\) as in Proposition~\ref{prop : lowbnd}. The only difference is that now
we specify that the distinguished positive element is always realized as \(1\).
Conversely,
arguing as in the proof of
Proposition~\ref{prop : reduction to cong_3,1}
it can be shown that \(\cong^{*}_\mathsf{ArGp}\)
is \(\mathbf{\Pi}^0_3\),
so that it is reducible to \(=^+_\R\) by Proposition~\ref{prop : HjoKecLou}.
Indeed, any isomorphism between two elements of \(X^*_\mathsf{ArGp}\) must match their distinguished elements, and from them we can recover the scalar determining the isomorphism as in Hion's lemma. Thus in equation~\eqref{eq:complexity} the first existential quatifier on \(k\) disappears, and the fraction \(\frac{b_k}{a_1}\) can be replaced by the constant ratio of the distinguished elements of the groups.
\end{proof}

\subsection{Invariants for \(\cong_{\mathsf{ArGp}}\)} \label{subsec : explicitinvariants}

Proposition~\ref{prop : reduction to cong_3,1} shows that \( \cong_{\mathsf{ArGp}} \) admits a complete classification using hereditarily countable sets of sets of reals (or, even more precisely, elements of \( \mathcal{P}_\ast^{3,1}(\mathbb{N}) \)) as invariants, but does not provide a constructive way to produce them. 
Below we explain how to explicitly compute such complete invariants for \(\cong_{\mathsf{ArGp}}\).
By H\"older's Theorem~\ref{theorem : Holder} it is enough to provide a classification for (countable) subgroups of \( \R \).

Given  $G\leq\mathbb{R}$ and $r\in\mathbb{R}\setminus\{0\}$, let $G/r=\set{g/r}{g\in G}$.
Define 
\begin{equation} \label{eq : A_G}
A_G=\set{G/r}{r\in G\setminus\{0\}}.
\end{equation}
We will often assume, without loss of generality, that $1\in G$, in which case $G\in A_G$.
Notice that \( G/r = G/(-r) \), hence  \( r \in G \) can be assumed to be strictly positive when considering \(G/r\).

\begin{prop}
\label{prop : invariants}
Let $G$ and $H$ be non-trivial subgroups of $\mathbb{R}$. Then $G$ and $H$ are order isomorphic if and only if $A_G=A_H$.
\end{prop}

\begin{proof}
Assume first $A_G=A_H$. In particular, there are positive elements $r\in G$ and $s\in H$ such that $G/r=H/s$.
Note that $G$ is order isomorphic to $G/r$, via the map $g\mapsto g/r$. Similarly, $H$ is order isomorphic to $H/s$, and therefore to $G$.

Assume now that $G$ and $H$ are order isomorphic.
By symmetry it suffices to show that $A_G\subset A_H$.
By Hion's Lemma~\ref{lem : Hion}, there is a positive real $\lambda$ such that the map $g\mapsto \lambda g$ is an isomorphism between $G$ and $H$.
Given any $r\in G\setminus\{0\}$, let $s \coloneqq \lambda r \in H$. Now
\begin{equation*}
    H/s=\set{h/s}{h\in H}=\set{\lambda g / \lambda r}{g\in G}=\set{g/r}{g\in G}=G/r,
\end{equation*}
so $G/r\in A_H$ for every \( r \in G \setminus \{ 0 \} \). We conclude that $A_G\subset A_H$.
\end{proof}


Proposition~\ref{prop : invariants} provides the desired classification for Archimedean groups. Given such a group it is enough to realize it as a subgroup \( G \) of \( \R \), and then assign the invariant \( A_G \) to it. Moreover, one easily realizes that such invariant can be turned into a \( \cong^*_{3,1} \)-like invariant, as explained in the following remark.


\begin{remark}
\label{rem : invariants}
After applying the usual identification between \( \R \) and \( \mathcal{P}(\N) \), one sees that for any countable $G\leq\mathbb{R}$  the invariant $A_G$ is a hereditarily countable set in $\mathcal{P}^3(\mathbb{N})$. Moreover, it can be completed to a $\cong^\ast_{3,1}$-invariant as follows: define $R\subset A_G\times A_G\times \mathbb{R}$ by setting $R(H,K,r)$ if and only if $r\in H\setminus\{0\}$ and $K=H/r$.
Then $R$ satisfies that for any $H,K \in A_G$ there is \( r \in H \subseteq \R \cap \mathrm{tc}(A_G) \) such that \( R(H,K,r) \), and for any $H,K,K'\in A_G$ and any $r$, if $R(H,K,r)$ and $R(H,K',r)$ both hold then $K=K'$.
\end{remark}

We finally notice that the above classification procedure is Borel in the codes. Indeed, given \(\vec{x} = (x_n)_{n \in \N} \in \mathcal{A}\), we define the invariant \(A_{\vec{x}} = \set{A^{\vec{x}}_n }{ 0 \neq n \in \N}\) where \(A^{\vec{x}}_n = \set{ \frac{x_k}{x_n} }{ k \in \N }\). (Recall that \(x_{0} = 0\) for all \(\vec{x}\in \mathcal{A}\).) 
 By Proposition~\ref{prop : invariants} we have \(\vec{x} \cong_{\mathcal{A}} \vec{y} \iff A_{\vec{x}} = A_{\vec{y}}\).
Then, following Remark~\ref{rem : invariants}, we define the ternary relation
\(R_{\vec{x}}\) on
\(A_{\vec{x}} \times A_{\vec{x}} \times (\R \cap \mathrm{tc}(A_{\vec{x}}))\) by declaring 
\[
(A_n,A_m,r) \in R_{\vec{x}} \iff r= \frac{a_m}{a_n}.
\]
Since the third component of \(R_{\vec{x}}\) is an element of \(\R \) and by definition \(\frac{a_m}{a_n} \in A_n \subseteq \tc(A_{\vec{x}})\), under the usual identification of \( \R \) with \( \mathcal{P}(\N ) \) the pair \((A_{\vec{x}}, R_{\vec{x}})\) can be construed as an element of  \(\mathcal{P}_{*}^{3,1}(\N)\), and is still an invariant for \(\cong_\mathcal{A}\).
It follows that the Borel map on \( X_{\mathsf{ArGp}} \) 
\begin{equation} \label{eq : AxG}
G \mapsto \vec{x}_G \mapsto A_{\vec{x}_G}
 \end{equation}
witnesses \({\cong_\mathsf{ArGp}} \leq_B {=^{++}_\R}\), while the Borel map
\begin{equation}
G \mapsto \vec{x}_G \mapsto (A_{\vec{x}_G}, R_{\vec{x}_G})
\end{equation}
witnesses  \({\cong_\mathsf{ArGp}} \leq_B {\cong^*_{3,1}}\).



\section{Irreducibility to $\cong^\ast_{3,0}$}
\label{sec : irred to 3,0}
In this section we prove Theorem~\ref{thm : main}, 
namely that $\cong_\mathsf{\mathrm{ArGp}}$
is not Borel reducible to $\cong^\ast_{3,0}$. In particular, it will follow that
\[
{=^+_\R} <_B {\cong_\mathsf{ArGp}} <_B {=^{++}_\R}.
\]
The proof uses methods from set theory, such as forcing over choiceless models of ZF set theory. We start by providing context with an informal overview of our strategy.

The technique we are going to describe was introduced by Shani~\cite{Sha}. Let \(V\) be the model of \(\mathsf{ZFC}\) we are working in. Suppose that \(E\) and \(F\) are Borel equivalence relations on standard Borel spaces \(X\) and \(Y\). Suppose further that we have a ``reasonable'' assignment of complete invariants \(x \mapsto A_x\) and \(y \mapsto B_y\) for \(E\) and \(F\), respectively. Here ``reasonable'' means that 
\begin{itemize}
\item
the invariants are hereditarily countable sets;
\item
the assignment is definable in \(V\) through a first-order formula \(\phi\) in the language of set theory;
\item
if \(M \supseteq N\supseteq V\) are models of \(\mathsf{ZF}\), then  \(\phi\) still defines a complete classification of \(E\) in both \(M, N\), and \(N\) and \(M\) compute (using \(\phi\)) the invariant \(A_x\) in the same way for every \(x \in X\) which belongs to \(N\).
\end{itemize}

We call such complete assignments of invariants \emph{absolute classifications by hereditarily countable sets}. Examples of such classifications that are relevant to this paper are:
\begin{itemize}
\item
the classification for \(\cong_\mathsf{ArGp}\) via the map
 \(G \mapsto A_{\vec{x}_G}\) from equation~\eqref{eq : AxG};
 \item
the classification for \(=^+_\R\) via the map \(y = (y_n)_{n \in \N} \mapsto B_y = \set{y_n}{ n \in \N}\).
\end{itemize}

%

Shani~\cite[Lemma~3.6]{Sha} proved that if \(E\) and \(F\) are as above and there is a Borel reduction \(f \colon X \to Y\) of \(E\) to F, then for every \(x \in X\) the invariants \(A_x\) and \(B_{f(x)}\) are ``interdefinable'' in the following technical sense:
\begin{itemize}
\item
\(V(A_x) = V(B_{f(x)})\), where \(V(A_x)\) is the minimal transitive model of \(\mathsf{ZF}\) containing the ground model \(V\) and \(A_x\), and similarly for \(V(B_{f(x)})\). 
\item
\(B_{f(x)}\) is definable in \(V(A_x)\) from \(A_x\) and parameters in \(V\) alone.
\end{itemize}
Moreover, the same remains true if we move to any generic extension of the ground model \(V\).

The above result yields a powerful technique to show that \(E \not\leq_B F\): if we can construct via forcing a new element \(x \in X\) such that its invariant \(A_x\) is more complicated than all invariants \(B_y\) for \(F\) (more precisely: \(V(A_x) \supsetneq V(B_y)\) for every \(B_y\)  definable in \(V(A_x)\) from \(A_x\) and parameters in \(V\) alone), then \(E \not\leq_B F\).

In this section, we will first apply this machinery to show that \(\cong_\mathsf{ArGp}\) is not Borel reducible to \(=^+_\R\) (Corollary~\ref{cor : not reducible =^+_R}), in fact replacing the former with \(\cong_{\mathcal{A}}\) (which, as shown in Section~3, is Borel bi-reducible with \(\cong_\mathsf{ArGp}\)). For such Borel equivalence relations, we consider the absolute classifications by hereditarily countable sets \(G \mapsto A_{\vec{x}_G}\) and \(y \mapsto B_y\) described above, although to simplify the notation the former will be presented just as \( G \mapsto A_G \), where \( A_G \) is as in equation~\eqref{eq : A_G}. Then we will design a suitable forcing \(\mathbb{P}\) so that any \(\mathbb{P}\)-generic filter \(F\) generates a subgroup \(G\) of \(\R\), and prove that \(V(A_G) \neq V(B)\) for any set of reals \(B \in V(A_G)\) that is definable from \(A_G \) and parameters in \(V\) alone. By virtue of the above discussion, this implies 
\[
{\cong_\mathsf{ArGp}} \sim_B {\cong_{\mathcal{A}}} \not \leq_B {=^+_\R}.
\] 

Using a similar strategy, we will then strengthen this result to \({\cong_\mathsf{ArGp} }\not\leq_B {\cong^*_{3,0}}\) (Theorem~\ref{thm : main}), exploiting the fact that, as noticed in Remark 3.7, any complete invariant for \(\cong^*_{3,0}\) can be described as a set of sets of reals which can be enumerated in a definable way using a parameter from the set itself. 

\subsection{Background}

For the rest of Section~\ref{sec : irred to 3,0} we assume a certain familiarity with set theory and, in particular, with the basics of the forcing technique. We begin with some specific background material that we will use, referring the reader to \cite{Jec03} for more on the subject.

Throughout this section \(V\) denotes a model of \(\mathsf{ZFC}\). Suppose $A$ is a set in some generic extension of $V$.
Then there exists a
\emph{minimal transitive model of} \(\mathsf{ZF}\) {containing $V$ and $A$}, denoted by $V(A)$.
All assertions below regarding Borel reductions are absolute. Therefore it suffices to consider $V=L$, G\"odel's constructible universe;
in this case $V(A)$ is the usual Hajnal relativized $L$-construction, $L(A)$. (Often times, as will be the case here, the axiom of choice fails in $V(A)$.)

Working in $V(A)$, we can consider the submodel of all sets which are hereditarily definable using $A$ and parameters from $V$. This is again an extension of $V$ which is a model of \( \mathsf{ZF} \) and contains $A$, and therefore is equal to $V(A)$ by minimality.
\begin{fact}
\label{fact;definability-V(A)}
The following holds in $V(A)$.
For any set $X$, there is some formula $\psi$, parameters $\bar{a}$ from the transitive closure of $A$ and $v\in V$ such that $X$ is the unique set satisfying $\psi(X,A,\bar{a},v)$.
Equivalently, there is a formula $\varphi$ such that $X=\set{x}{\varphi(x,A,\bar{a},v)}$.
In this case we say that $\bar{a}$ is a \emph{definable support} for $X$.
\end{fact}
We will be particularly interested in sets with empty definable support, that is, those definable from $A$ and parameters in $V$ alone.


%
%



As discussed above, a key ingredient of our technique is the following:

\begin{thm}[Shani~{\cite[Lemma~3.6]{Sha}}]\label{thm : symmetric models irreducibility}
Suppose that \(E\) and \(F\) are Borel equivalence relations on standard Borel spaces \(X\) and \(Y\), respectively, and \(x\mapsto A_x\) and \(y\mapsto B_y\) are absolute classifications by hereditarily  countable sets. Assume further that $f\colon X\to Y$ is a Borel reduction of $E$ to $F$.
Let \(x\) be an element of \( X \) in some generic extension of \( V \), and let \(A = A_x\) and \(B = B_{f(x)}\).
Then \(V(A) = V(B)\), and $B$ is definable in $V(A)$ from $A$ and parameters in $V$ alone.
\end{thm}

Specializing the above theorem to our situation we get the following result.

\begin{cor}
\label{cor : irred to cong}
Suppose $E$ is a Borel equivalence relation on a standard Borel space \( X \), and $x\mapsto A_x$ is an absolute classification of $E$ by hereditarily countable sets.
Let \(x\) be an element of \( X \) in some generic extension of \( V \), and set $A=A_x$.
\begin{enumerate}
    \item If $E\leq_B {=^+_\R}$ then there is a set of reals $B\in V(A)$ such that $B$ is definable from $A$ and parameters in $V$ alone and $V(A)=V(B)$.
    \item If $E\leq_B {\cong^\ast_{3,0}}$ then there is a set of sets of reals $B\in V(A)$ such that $B$ is definable from $A$ and parameters in $V$ alone, $V(A)=V(B)$, and $B$ is countable (in $V(A)$).
\end{enumerate}
\end{cor}

\begin{proof}[Sketch of the proof]
Part (1) follows immediately from Theorem~\ref{thm : symmetric models irreducibility}, as $=^+_\R$ admits an absolute classification by hereditarily countable sets  with sets of reals as complete invariants.

Part (2) is \cite[Claim 6.4]{Sha}. The point is that an invariant for $\cong^\ast_{3,0}$ is a pair $(D,R)$ where $D$ is a set of sets of reals, and $R$ codes enumerations of $D$ (see Remark~\ref{rmk : invs for 3,0}).
Since $R$ is contained in $D\times D\times\mathbb{N}$, it is countable as well.
Take a set of sets of reals $B$ coding the pair $(D,R)$ via a definable injective map from $\mathcal{P}(\mathbb{R})\times \mathcal{P}(\mathbb{R})\times \mathcal{P}(\mathbb{R})\times\mathbb{N}$ into $\mathcal{P}(\mathbb{R})$.
Then $B$ is countable, and $B$ and $(D,R)$ are interdefinable in a simple manner.
\end{proof}

\begin{remark}
Readers unfamiliar with the expression ``\( x \) is an element of \( X \) in some generic extension on  \( V \)'' may concentrate on the unique specific situations considered in this paper, namely, when \( X \) is the standard Borel space of countable Archimedean groups, and \( E \) is either order isomorphism or bi-embeddability on \( X \). Then the above condition on \( x \) appearing in both Theorem~\ref{thm : symmetric models irreducibility} and Corollary~\ref{cor : irred to cong} reads as follows: ``Let \( x \) be a countable Archimedean group in some generic extension of \( V \).''
\end{remark}


\subsection{The basic Cohen model}

In this subsection we present some standard results about the basic Cohen models that we will use later.
This model is one of the first examples presented by Paul Cohen for the independence of the axiom of choice from the axioms of \(\mathsf{ZF}\), that is, a model of \(\mathsf{ZF}\) set theory in which the axiom of choice fails (see \cite{Kan08}).  
For the reader's convenience, we include here self-contained full proofs of such results. Our presentation is slightly different from what can be found in the standard literature \cite{Fel71,Jec73, Jec03} and is more close to the ``definability arguments'' that we will use later. Readers familiar with these concepts may safely skip this section.



Let $\mathbb{P}$ be the poset of all finite functions $p$ with $\dom p\subset \mathbb{N}$ and $p(i)\subset\mathbb{R}$ is an open interval with rational endpoints.
Say that $p$ \emph{extends} $q$ in $\mathbb{P}$ (in symbols, $p\leq q$) if $\dom p\supset \dom q$ and $p(i)\subset q(i)$ for each $i\in\dom q$.
Suppose $G\subset\mathbb{P}$ is generic over $V$. Define $x(n)\in \mathbb{R}$ as the unique real such that $x(n)\in p(n)$ for all $p\in G$.
Then $x=\seqq{x(n)}{n\in\mathbb{N}}$ is a generic sequence of generic Cohen reals.
Let $A^{1}=\set{x(n)}{n\in\mathbb{N}}$ be the unordered set%
\footnote{The notation \( A^1 \) is consistent with the one of Monro~\cite{Mon73}.}
 enumerated by $x$, that is, the $=^+_\R$-invariant of $x\in\mathbb{R}^{\mathbb{N}}$.
The basic Cohen model can then be presented as $V(A^{1})$.

A standard forcing argument shows:
\begin{fact}\label{claim;Cohen-set-dense}
$A^{1}$ is dense in $\mathbb{R}$.
\end{fact}

In particular, $A^{1}$ is infinite. Recall the following definition:

\begin{defn}
A set \(X\) is \emph{Dedekind-finite} if it has no countably infinite subset. Equivalently, there are no infinite sequences in \(X\).
\end{defn}

Below we show that $A^{1}$ is Dedekind-finite in $V(A^{1})$.
The key idea is that the reals in $A^{1}$ are ``sufficiently indiscernible'' over the ground model $V$. This is captured by the following.

\begin{lemma}[Continuity Lemma, {\cite[p.133]{Fel71}}]\label{lem;continuity-Cohen-model}
Let $\phi$ be a formula, $\bar{a}=a_0,...,a_{n-1}$ a finite sequence of distinct members of $A^{1}$, and $v\in V$. Suppose $\phi(A^{1},\bar{a},v)$ holds in $V(A)$.
Then there are open intervals with rational endpoints $U_0,...,U_{n-1}$ such that $a_i\in U_i$ and for any $\bar{b} = b_0,...,b_{n-1}$ consisting of distinct elements from $A^{1}$, if $b_i\in U_i$ for all \( i \leq n-1 \), then $\phi(A^{1},\bar{b},v)$ holds in $V(A^{1})$.
\end{lemma}

\begin{proof}
Work in the \( \mathbb{P} \)-generic extension $V[G]$.
Fix $l_0,...,l_{n-1}$ such that $a_{i}=x(l_i)$ for all \( i \leq n-1 \).
By assumption $\phi^{V(A^{1})}(A,x(l_0),...,x(l_{n-1}),v)$ holds in $V[G]$, where $\phi^{V(A^{1})}$ is the formula $\phi$ relativized to the model $V(A^{1})$ (see \cite{Jec03} for relativization).

By the forcing theorem, there is a condition $p\in G$ forcing that $\phi^{V(\dot{A^{1}})}(\dot{A},\dot{x}(l_0),...,\dot{x}(l_{n-1}),\check{v})$ holds.
Define $U_i=p(l_i)$.

Let $b_0,...,b_{n-1}$ be distinct reals such that $b_i\in U_i\cap A^{1}$.
Fix $t_0,...,t_{n-1}$ such that $b_i=x(t_i)$.
Let $\pi$ be a finite support permutation of $\mathbb{N}$ such that
\begin{enumerate}
    \item $\pi(t_i)=l_i$ for $i=0,...,n-1$;
    \item for any $m$ in the domain of $p$, $x(\pi^{-1}(m))\in p(m)$.
\end{enumerate}
If $m=l_i$ for some $i$, part (2) is guaranteed by (1) and by the assumption that $x(t_i) = b_i \in U_i$. For other values of $m$, part (2) can be guaranteed by the density of $A^{1}$.

Any such permutation $\pi$ of $\mathbb{N}$ induces an automorphism of the poset $\mathbb{P}$ as follows. For $r\in\mathbb{P}$, $m$ is in the domain of $\pi r$ if and only if $\pi^{-1}(m)$ is in the domain of $r$, in which case $\pi r(m)=r(\pi^{-1}(m))$.
That is, $\pi$ permutes the reals enumerated by the generic $x$.
Define $G'=\set{\pi r}{r\in G}$. Then $G'\subset\mathbb{P}$ is generic over $V$ and such that:
\begin{itemize}
    \item $\dot{x}[G'](l_i)=\dot{x}[G](t_i)=x(t_i)=b_i$ for $i=0,...,n-1$;
    \item $\dot{A^{1}}[G']=\dot{A^{1}}[G]=A^{1}$;
    \item $p\in G'$.
\end{itemize}
Since $p$ forces $\phi^{V(\dot{A^{1}})}(\dot{A^{1}},\dot{x}(l_0),...,\dot{x}(l_{n-1}),\check{v})$, we conclude that $\phi^{V(A^{1})}(A^{1},b_0,...,b_{n-1},v)$ holds in $V[G']$, and therefore $\phi(A^{1},b_0,...,b_{n-1},v)$ holds in $V(A^{1})$, as desired.
\end{proof}

\begin{cor}[See \cite{Jec73}]
In $V(A^{1})$, the set $A^{1}$ is Dedekind-finite.
\end{cor}
\begin{proof}
Suppose to the contrary that there is an infinite sequence $y=\seqq{y(n)}{n\in\mathbb{N}}$ of distinct members of $A^{1}$.
By Fact~\ref{fact;definability-V(A)} there is some formula $\psi$, a finite sequence $a_0,...,a_{n-1}$ of distinct reals from $A^{1}$, and some $v\in V$ such that $y$ is the unique solution to $\psi(y,A^{1},a_0,...,a_{n-1},v)$ in $V(A^{1})$.
Fix $k\in\mathbb{N}$ such that $y(k)\neq a_i$ for $i=0,...,n-1$. Let $a=y(k)$.

Consider the formula $\theta(A^{1},a_0,...,a_{n-1},a,v)$ saying ``for the unique $y$ such that $\psi(y,A^{1},a_0,...,a_{n-1},v)$ holds, $a=y(k)$''.
By the above lemma, there is an open set $U\subset \mathbb{R}$ such that for any $b\in U\cap A^{1}$, if $b$ is different from $a_0,...,a_{n-1}$ then $\theta(A^{1},a_0,...,a_{n-1},b,v)$ holds.
By Fact~\ref{claim;Cohen-set-dense} there is a $b\in U\cap \mathbb{R}$ such that $b\notin \{a_0,...,a_{n-1},a\}$.
It follows that $b=y(k)=a$, a contradiction.
\end{proof}


\subsection{A complicated invariant for \(\cong_\mathsf{ArGp}\)} \label{subsec : complicatedinvariant}

Before proving that ${\cong_{\mathsf{ArGp}}}\not\leq_B {=^+_\R}$ we explain some natural difficulties which will shape the actual proof. 
To show ${\cong_{\mathsf{ArGp}}}\not\leq_B {=^+_\R}$, we want to find an invariant of the form $A \coloneqq A_G$, in some generic extension such that $V(A)$ cannot be presented as $V(B)$, where $B$ is a set of reals (i.e.\ a $=^+_\R$-invariant).

Work over the ``basic Cohen model'' $V(A^1)$.
If we take our group $G$ to be $K = \Q(A^1)$, the field generated by $A^1$, then the invariant of $K$ is the singleton $A_K=\{K\}$, which is simply a set of reals and thus too simple for our purposes.
On the other hand, consider the
\(\Q\)-vector space \(H = \spann_\Q  A^1\cup\{1\}\). Then for distinct $a,b\in H\setminus\{0\}$, either $H/a \cap H/b = \Q$ or $H/a = H/b$. Then, up to this fixed set $\mathbb{Q}$, the invariant $A_H$ consists of a set of pairwise disjoint sets of reals. Since any set of pairwise disjoint sets of reals can be definably coded as a single set of reals (e.g., see \cite[Proposition 5.1]{Sha}), the invariant $A_H$ is again as simple as a set of reals.

Instead, we shall add via forcing a generic group $G$ between $H$ and $K$. In our argument it will be crucial
to add such $G$ by a forcing with enough automorphisms that change $G$ while preserving the invariant $A_G$. It will also be important to have such automorphisms densely, that is, given a condition $p$ compatible with the generic group $G$, we want to have an automorphism, sending $G$ to $G'$, which preserves the invariant $A_{G'}=A_G$ and such that $p$ is compatible with $G'$ as well.

These automorphisms will be those taking a generic group $G$ to $G/r$ for some $r\in G\setminus\{0\}$. 
This looks fine so far: suppose we force with finite approximations. Say $p$ is a condition which determines that $a\in G$ and $b\notin G$ for $a,b\in A^1$. Then we can take some distinct $r\in A^1$, extend $p$ to $q$ which decides that $a\cdot r\in G$ and $b\cdot r\notin G$. Now $q$ forces that after applying the automorphism sending $G$ to \( G/r \), we get that $G'=G/r$ also agrees with $p$ (it decides that $a\in G'$ and $b\notin G'$).

In order for the forcing to be closed under such automorphisms sending $G$ to $G/r$, we must allow the conditions to have in their domain arbitrary negative (and positive) powers of elements from $A^1$. For example, we may want to choose $G$ as generically as possible from the field $K$.

It will also be important to add the group $G$, generically over $V(A^1)$, without adding any real.
This will ensure  that sets of reals will be subsets of the ground model $V(A^1)$, and therefore easier to analyze. 
However, if we add $G$ as a ``generic additive subgroup of $K$'', by finite approximations, then for any $a\in A^1$, the set $\set{n}{a^n\in G}$ would be a new real.
To avoid that, we will add $G$ so that arbitrary powers $a^n$ can occur, but for any fixed $a$, $\set{n\in \mathbb{Z}}{a^n\in G}$ is bounded.

We now provide the actual definition of a forcing \(\mathbb{P}\) satisfying all restrictions described so far, and prove that it has the desired properties.

For any distinct $a_1,...,a_k\in A^1$, let $D(a_1,...,a_k)$ be the set of all elements of the form $a_1^{l_1}\cdot...\cdot a_k^{l_k}$ where $l_1,...,l_k$ are integers.
By convention $D(\emptyset)=\{1\}$.
For a natural number $m\geq 1$, let
$D^m(a_1,...,a_k)$ be the subset of $D(a_1,...,a_k)$  consisting of all elements of the form $a_1^{l_1}\cdot...\cdot a_k^{l_k}$ with $-m\leq l_1,...,l_k\leq m$.
Note that $D(a_1,...,a_k)=\bigcup_m D^m(a_1,...,a_k)$.

\begin{defn} \label{def : forcinggroup}
Let $\mathbb{P}$ be the poset of conditions of the form $p=(S_p,G_p)$ where $S_p$ is a finite subset of $A^1$, $G_p$ is a subset of $D^m(S_p)$ for some $m\geq 1$, and $1\in G_p$. For $p,q\in\mathbb{P}$, say that $q$ extends $p$ if
$S_q\supset S_p$ and $G_q\cap D(S_p)= G_p$. 
\end{defn}

Clearly $\mathbb{P}$ is a partial order with maximal element $(\emptyset,\{1\})$. We isolate the following two properties which follow from the mutual genericity of the reals in $A^1$.

\begin{fact}\label{fact : basic properties}
\begin{enumerate}
    \item $D(S_1)\cap D(S_2)=D(S_1\cap S_2)$ for any finite $S_1,S_2\subset A^1$.
    \item Given conditions $p,q$ such that $G_p$ and $G_q$ agree on $S_1\cap S_2$, $p$ and $q$ are compatible.
\end{enumerate}
\end{fact}

As expected, the conditions $p=(S_p,G_p)$ are an approximation of a generic group $\dot{G}$ we wish to add. Moreover, each \( p \in \mathbb{P} \) completely decides $\dot{G}\cap D(S_p)$, and it decides it to be a subset of $D^m(S_p)$ for some $m$.

The generic group \( G \) is obtained from any \( \mathbb{P} \)-generic (over \( V(A^1) \)) \( F \) as follows.
Let $F'$ be the union of all $G_p$ where $p \in F$.
We let $G = \langle F'\rangle$ be the additive subgroup of $\mathbb{R}$ generated by $F'$, and let $A \coloneqq A_G=\set{G/r}{r\in G\setminus\{0\}}$ be its invariant.
Note that $V(A)=V(A^1)[F]$.


For $a\in A^1$ define $\mathbb{P}_a=\set{p\in\mathbb{P}}{a\in G_p}$ and $\mathbb{P}_{a^{-1}}=\set{p\in\mathbb{P}}{a^{-1}\in G_p}$. Note that by mutual genericity, $a\in S_p$ for every $p \in \mathbb{P}_a, \mathbb{P}_{a^{-1}}$.
Define $\delta_a\colon\mathbb{P}_a\to\mathbb{P}_{a^{-1}}$ by $\delta_a(p)=(S_p,G_p/a)$. Define $\pi_a\colon\mathbb{P}_{a^{-1}}\to\mathbb{P}_{a}$ by $\pi_a(p)=(S_p,G_p\cdot a)$.
Note that for any $p\in \mathbb{P}_a$, the forcings \( \mathbb{P} \) and \( \mathbb{P}_a \) coincide on conditions extending \( p \),
and similarly for $\mathbb{P}_{a^{-1}}$.

\begin{fact}
The maps $\delta_a$ and $\pi_a$ are isomorphisms between $\mathbb{P}_a$ and $\mathbb{P}_{a^{-1}}$.
\end{fact}

We will use $\delta_a$ to change the generic group $G$ without changing its invariant \( A_G \).
For a condition $p\in\mathbb{P}$ and $S\subset S_p$, define $p\restriction S=(S,G_p\cap D(S))$.

\begin{lemma}\label{lem;no-new-reals}
Forcing with $\mathbb{P}$ over $V(A^1)$ does not add any subsets of $V$.
\end{lemma}

\begin{proof}
Let $\tau\in V(A^1)$ be a $\mathbb{P}$-name for a subset of $V$. Then
$\tau$ is definable in $V(A^1)$ using $A^1$, a finite tuple $\bar{a}_0\subset A^1$ and a parameter $v\in V$.
We will show that for any $w\in V$, if $p\force \check{w}\in\tau$, then $p\restriction\bar{a}_0\force \check{w}\in\tau$. It follows that any condition with domain $\bar{a}_0$ already decides the statement ``$\check{w}\in \tau$'' for all $w\in V$. Thus the set added by $\tau$ is forced to be in the ground model $V(A^1)$.

Assume that $p\in\mathbb{P}$ forces that $\check{w}\in\tau$. Let $\bar{a}_1$ enumerate $S_p$. We may assume that $\bar{a}_0$ is contained in $\bar{a}_1$ and write $\bar{a}_1=\bar{a}_0^\frown \bar{a}$ where $\bar{a}$ is a finite subset of $A^1$ disjoint from $\bar{a}_0$. Let $\bar{a}=a_1,...,a_l$.

There are basic open sets $U_1,...,U_l$ such that $a_i\in U_i$ for \( 1 \leq i \leq l \), and for any $\bar{a}'\subset A^1$, if $a'_i\in U_i$ for all $1 \leq i \leq l$ then $p[\bar{a}']\force \check{w}\in\tau$, where $p[\bar{a}']$ is defined from $p$ by replacing $a_i$ with $a'_i$.
This follows from the Continuity Lemma~\ref{lem;continuity-Cohen-model} applied to the statement ``$\check{w}\in\tau$''. 
Note that $\mathbb{P}$ is defined from $A^1$ alone, therefore so is the associated forcing relation $\force$.
The name $\tau$ is definable from $A^1$, $\bar{a}_0$ and $v$, and $p$ involves only the parameters $\bar{a}_1$. Thus the statement ``$\check{w}\in\tau$'' can be written as $\phi(A^1,\bar{a}_0,\bar{a},v,w)$ for some formula $\phi$.

Suppose now $q$ is some condition extending $p\restriction\bar{a}_0$.
Let $\bar{a}'$ be disjoint from $S_q$ and such that $a'_i\in U_i$.
Then $q$ and $p[\bar{a}']$ are compatible.
We established that any extension of $p\restriction\bar{a}_0$ is compatible with a condition forcing $\check{w}\in\tau$. It follows that $p\restriction\bar{a}_0\force \check{w}\in\tau$, as required.
\end{proof}

It follows from Lemma~\ref{lem;no-new-reals} that $\mathbb{P}$ adds no new members to any Polish space, as such members can be coded as subsets of V. For example, no new members of $\mathbb{R}$ or $\mathbb{R}^\mathbb{N}$ are added.

\begin{remark}
In particular $\mathbb{P}$ adds no sets of ordinals over $V(A^1)$. This phenomenon can only happen when forcing over models in which the axiom of choice fails. 
An extension of the basic Cohen model which adds no sets of ordinals was first constructed in \cite{Jec71}. Higher rank analogues were established in \cite{Mon73}, \cite{Kar18}, \cite{Kar19} and \cite{Sha}.
Extensions adding no sets of ordinals to $\mathrm{L}(\mathbb{R})$ were studied in \cite{HMW} and \cite{DH21}.
\end{remark}

\begin{lemma}\label{lem;genericity}
Given any $p=(S_p,G_p)\in\mathbb{P}$ and a finite set $S\subset A^1$, the collection of all $q\in\mathbb{P}$ such that there is some $d\in A^1\setminus S$ for which $d\in G_q$ and $G_q\cap (d\cdot D(S_p))=d\cdot G_p$ is dense.
Note that in this case $\delta_d(q)$ extends $p$.
\end{lemma}
\begin{proof}
Let $r=(S_r,G_r)\in\mathbb{P}$ be a condition. Take $d\in A^1\setminus (S_r\cup S_p\cup S)$.
Then $d\cdot D(S_p)\cap D(S_r)=\emptyset$ by Fact~\ref{fact : basic properties}.
Define $q=(S_q,G_q)$ so that $S_q=S_r\cup S_p\cup\{d\}$, $G_q\cap (d\cdot D(S_p))=d\cdot G_p$, and $G_q\cap D(S_r)= G_r$. Then $q$ extends $r$ and is as required.
\end{proof}
\begin{cor}\label{cor : genericity inf often}
If $F$ is a $\mathbb{P}$-generic filter over $V(A^1)$ and $p\in F$, then there are infinitely many values $d\in A^1$ for which there is $q\in F$ such that $d\in G_q$ and $\delta_d(q)$ extends $p$.
Note that then $\delta_d(F)$ is a generic filter containing $p$.
\end{cor}
\begin{proof}
Fix $p\in F$. For any finitely many $d_1,...,d_k$, it follows from the previous lemma that there is a $q\in F$ and $d\in A^1\setminus\{d_1,...,d_k\}$ as desired.
\end{proof}

\begin{lemma}\label{claim : genericity}
Given distinct $a,b\in A^1$, $G/a$ and $G/b$ are distinct
\end{lemma}

\begin{proof}
Define $p=(S_p,G_p)$ by $S_p=\{a,b\}$ and $G_p=\{a\}$.
By Lemma~\ref{lem;genericity} there is some $q$ in the generic $F$ for which there is some $d\in A^1$ such that $G_q\cap(d\cdot D(\{a,b\}))=\{d a\}$.
It follows that $da\in G$ and $db\notin G$, thus $d\in (G/a)\setminus (G/b)$.
\end{proof}


\subsection{Proof of Theorem~\ref{thm : main}}\label{sec : proof of main thm}
As a warm up, we first prove that $\cong_{\mathsf{ArGp}}$ is not Borel reducible to $=^+_\R$.
Recall that \( A \coloneqq A_G \) is the invariant of the generic group \( G = \langle F' \rangle \) added by forcing with \( \mathbb{P} \) over \( V(A^1) \).

\begin{lemma}\label{lem : not generated by set of reals}
If $B\in V(A)$ is a set of reals which is definable from $A$ and parameters in $V$ alone, then $B\in V(A^1)$. Thus, \(V(A) \neq V(B)\).
\end{lemma}

\begin{proof}
Let $\phi$ be a formula  and \( v \in V \) be such that, in $V(A)$, $B$ is defined as 
\[
B=\set{x\in\mathbb{R}}{\phi(x,A,v)}.
\]
Since $\mathbb{P}$ does not add reals, $B\subset V(A^1)$. We claim that $B$ can be defined in $V(A^1)$ as the set of all $x\in\mathbb{R}$ such that $\mathbb{P}\force\phi^{V(\dot{A})}(\check{x},\dot{A},v)$.
To this aim, it suffices to show that two conditions in $\mathbb{P}$ cannot force conflicting values of $\phi^{V(\dot{A})}(\check{x},\dot{A},v)$.

Indeed, suppose $p$ and $q$ force $\phi^{V(\dot{A})}(\check{x},\dot{A},v)$ and $\neg\phi^{V(\dot{A})}(\check{x},\dot{A},v)$, respectively.
Let $F_1$ be a \( \mathbb{P} \)-generic filter extending $p$. By Lemma~\ref{lem;genericity} there is some $d\in A^1$ such that $F_1\cap d\cdot D(S_q)=d\cdot G_q$. Let $F_2=\delta_d(F_1)$. Then $F_2$ is a $\mathbb{P}$-generic filter extending $q$.
Finally, $\dot{G}[F_2]=\dot{G}[F_1]/d$: any element in $\dot{G}[F_1]$ is some linear combination $u=\sum_{i<k}g_i$ where $g_i\in F_1$. So $\frac{u}{d}=\sum_{i<k}\frac{g_i}{d}\in \dot{G}[F_2]$, and vice versa.
It follows that $\dot{A}[F_1]=\dot{A}[F_2]=A$.
This is a contradiction, since working in $V(A^1)[F_1]$ we conclude that $\phi^{V(A)}(x,A,v)$ holds, yet working in $V(A^1)[F_2]$ we conclude that $\phi^{V(A)}(x,A,v)$ fails.
\end{proof}

From Lemma~\ref{lem : not generated by set of reals} and the contrapositive of Corollary~\ref{cor : irred to cong}(1) we get:

\begin{cor}
\label{cor : not reducible =^+_R}
\( {\cong_\mathsf{ArGp}} \not\leq_B {=_\R^{+}} \).
\end{cor}

\begin{remark}
Technically speaking, Corollary~\ref{cor : irred to cong} should be applied in a further generic extension of  \( V(A) = V(A^1)[F] \) where we add a countable enumeration \( \vec{x}_G \) of \( G \) by forcing with finite conditions. This is because then \( A = A_G = A_{\vec{x}_G} \) properly becomes an invariant of \( \cong_\mathsf{ArGp} \) according to Section~\ref{subsec : explicitinvariants}. Notice that Lemma~\ref{lem : not generated by set of reals} still holds in such generic extension because it just refers to assertions concerning \( V(A) \), which is not affected by the further forcing. Similar considerations apply to the proof of Theorem~\ref{thm : main} below and to the proof of Proposition~\ref{prop : bi-embed non reduction}. 
\end{remark}


We now move to the proof of Theorem~\ref{thm : main}.

\begin{proof}[Proof of Theorem~\ref{thm : main}]
By the contrapositive of Corollary~\ref{cor : irred to cong}(2), it is enough to prove that $V(A)\neq V(B)$ whenever $B\in V(A)$ is a set of sets of reals definable from $A$ and parameters in \( V \) alone such that \( B \) is countable in \( V(A) \).

Assume towards a contradiction that $V(A)=V(B)$. In particular $G\in V(B)$ and therefore is definable as the unique solution to $\phi(G,\bar{U},\bar{x},v)$, where $\bar{U}=U_1,...,U_k$ are finitely many members of $B$, $\bar{x}=x_1,...,x_l$ are finitely many reals in the transitive closure of $B$, and $v\in V$.
Fix a condition $p=(S_p,G_p)$ in the generic forcing the above.
Note that $\bar{x}$ is in the ground model $V(A^1)$ by Lemma~\ref{lem;no-new-reals}.

By Corollary~\ref{cor : genericity inf often} there are infinitely many $d\in G$ such that the generic $G/d=\delta_d(G)$ extends $p$. Since $p$ forces that $\dot{G}$ is defined in $V(A)$ as the unique solution to $\phi(G,\dot{\bar{U}},\check{\bar{x}},v)$ it follows that there is some $\bar{U}_d$ from $B$ such that $G/d$ is the unique solution in $V(A)$ to $\phi(G/d,\bar{U}_d,\bar{x},v)$. (This follows from similar arguments as in Lemma~\ref{lem : not generated by set of reals}. Note that $G$ and $G/d$ calculate the same $A$ and so the same $B$.) 

Working in $V(A)$, $B$ is countable and so is $B^k$. Let $\bar{U}_i$, $i<\omega$ be an enumeration of the elements $\Tilde{\bar{U}}$ in $B^k$ for which there exists $d\in A^1$ such that $G/d$ is defined by $\phi(G/d,\Tilde{\bar{U}},\bar{x},v)$.
For any $i<\omega$, let $d_i$ be the unique $d\in A^1$ such that $\phi(G/d,\bar{U}_i,\bar{x},v)$ holds (uniqueness follows from Lemma~\ref{claim : genericity}).
Then $\seqq{d_i}{i<\omega}$ is a real, and therefore is in $V(A^1)$ by Lemma~\ref{lem;no-new-reals}. However, in the Cohen model $V(A^1)$ there are no infinite sequences in $A^1$, a contradiction.
\end{proof}
 
 The following question remains open.
 
 \begin{question}
Is \(\cong^*_{3,1}\) (or even \(\cong^*_{3,0}\)) Borel reducible to \(\cong_\mathsf{ArGp}\)?
\end{question}

\section{Embeddability on countable Archimedean groups} \label{sec : embeddability on ArGp}

Denote by \( \embeds_{\mathsf{ArGp}} \) the embeddability relation on \( X_{\mathsf{ArGp}} \), namely, for \( G,H \in X_\mathsf{ArGp} \) set \( G \embeds_\mathsf{ArGp} H \) if and only if there is an order preserving homomorphism from \( G \) to \( H \). (Recall that since the linear orders on our groups are strict, any such homomorphism is automatically an embedding.)
By the discussion following Proposition~\ref{prop : continuous}, the relation \( \embeds_{\mathsf{ArGp}} \) is a Borel quasi-order, that is, a reflexive and transitive relation which is Borel as a subsets of \( X_\mathsf{ArGp} \times X_\mathsf{ArGp} \). Generalizing the notion of Borel reducibility to quasi-orders one gets a way to assess the complexity of \( \embeds_{\mathsf{ArGp}} \), and consequently of the associated bi-embeddability relation \( \equiv_{\mathsf{ArGp}} \).

\begin{defn}
Given binary relations \(P, Q\) on standard Borel spaces \(X, Y\), respectively, a map \(f \colon X \to Y\) is said to be a \emph{reduction} from \(P\) to \(Q\) if for all \(x, y \in X\),
\[
x \mathrel{P} y \iff f(x) \mathrel{Q} f(y) .
\]
When there is a Borel reduction from \( P \) to \( Q \), we say that \(P\) is \emph{Borel reducible} to \(Q\) and write \( P \leq_B Q \). The bi-reducibility relation \( \sim_B \) is defined accordingly.
\end{defn}

If \( P \) is a quasi-order on \( X \) we denote by \( E_P \) the canonical equivalence relation induced by \( P \), namely, \( x \mathrel{E_P} y \) if and only if both \( x \mathrel{P} y \) and \( y \mathrel{P} x \). It is clear that if \( P \leq_B Q \) then \( E_P \leq_B E_Q \). (But the converse might fail, depending on \( P \) and \( Q \).)

Rosendal~\cite{Ros} introduced a  jump operator for quasi-orders which is the analogue, in that context, of the Friedman-Stanley jump. If \( P \) is a quasi-order on a standard Borel space \(X \), we define its jump \( P^{\mathrm{cf}} \) on the space \( X^\N \) setting 
\[ 
(x_i)_{i \in \N} \mathrel{P^\mathrm{cf}} (y_i)_{i\in \N} \iff \forall i \in \N \, \exists j \in \N \, (x_i \mathrel{P} y_j).
 \] 
Rosendal's jump is the asymmetric version of the Friedman-Stanley. In fact, if \( E \) is an equivalence relation, then \( E^+\) coincides with \( E_{E^\mathrm{cf}} \).

Starting with \( =_\R \), we can define a hierarchy of Borel quasi-orders \( =_\R^{\alpha\text{-}\mathrm{cf}} \), \(\alpha < \omega_1 \), by applying Rosendal's jump at successor stages and the usual countable product at limit stages. For example, the quasi-order \( =^{\mathrm{cf}}_\R \) is the first element in the hierarchy after \(=_\R \) and is basically the inclusion relation between countable sets of reals, so that the canonical equivalence relation associated to it is exactly \( =_\R^+ \), that is,
\[
{=_\R^+} = E_{=_\R^\mathrm{cf}}. 
\]
By~\cite[Proposition 6 and Corollary 15]{Ros}, the mentioned hierarchy is strict (that is, \( {=_\R^{\alpha\text{-}\mathrm{cf}}} <_B {=_\R^{\alpha+1\text{-}\mathrm{cf}}} \) for all countable \( \alpha \)) and \( \leq_B \)-cofinal among all Borel quasi-orders. 
The arguments developed so far allow us to locate the Borel quasi-order \( \embeds_\mathsf{ArGp} \) in such hierarchy and prove Theorem~\ref{thm : main-embed}.

\subsection{Bounds for \( \embeds_\mathsf{ArGp} \) and \( \equiv_\mathsf{ArGp} \)}

First we consider lower bounds. As for the isomorphism relation, when useful we can replace \( \embeds_\mathsf{ArGp} \) and \( \equiv_\mathsf{ArGp} \) with the corresponding relations \( \embeds_\mathcal{A} \) and \( \equiv_\mathcal{A} \) on the coding space \( \mathcal{A} \). In fact, by the discussion in Section~\ref{sec:boundsforisomorphism} (and in particular Proposition~\ref{prop : continuous}) we get
\[ 
{\embeds_\mathsf{ArGp}} \sim_B {\embeds_\mathcal{A}} \qquad \text{and} \qquad {\equiv_\mathsf{ArGp}} \sim_B {\equiv_\mathcal{A}}.
 \]

\begin{prop} \label{prop : incomparable}
\( {=_\R^\mathrm{cf}} \leq_B {\embeds_\mathsf{ArGp}} \), thus \(  E_{=_\R^\mathrm{cf}} \leq_B {\equiv_\mathsf{ArGp}} \).
\end{prop}

\begin{proof}
We argue as in the proof of Proposition~\ref{prop : lowbnd}. By Proposition~\ref{prop : fields} we obtain that  if the sequences \( (x_n)_{n \in \N}, (y_n)_{n \in \N} \in \R^\N \) enumerate two sets \(A,B \subseteq \R \), respectively, then \( A \subseteq B \) if and only if 
\(  \Q(\set{f(x_n)}{n \in \N}) \) embeds into \(  \Q(\set{f(y_n)}{n \in \N}) \).
This shows \( {=_\R^\mathrm{cf}} \leq_B {\embeds_\mathcal{A}} \), and by symmetrization we also get \( E_{=_\R^\mathrm{cf}} \leq_B {\equiv_\mathcal{A}} \).
\end{proof}

\begin{remark}
\label{rk : incomparable}
In particular, by (the proof of) Proposition~\ref{prop : incomparable} there are infinitely many countable divisible Archimedean groups which are pairwise incomparable with respect to embeddability, a fact that will be crucially used in the proof of Proposition~\ref{prop : CLO and ODAG}.
\end{remark}


As for upper bounds, we argue as in Section~\ref{subsec : explicitinvariants} using again the assignment \( G \mapsto A_G \) defined in equation~\eqref{eq : A_G}.

\begin{lemma} \label{lem : emb}
Let $G,H$ be subgroups of $\mathbb{R}$. Then there is an order preserving homomorphism from $G$ to $H$ if and only if for every $X\in A_G$ there exists $Y\in A_H$ such that $X\subseteq Y$.
\end{lemma} 

\begin{proof}
Suppose \( \phi \colon G \to H \) is an order preserving homomorphism. By Hion's lemma~\ref{lem : Hion} there is \( \lambda \in \R^+ \) such that \( \phi(g) = \lambda g \) for all \( g \in G \).
Given $X\in A_G$, let \( r \in G \setminus \{ 0 \} \) be such that  $X = G/r$. In particular, \( \lambda r \in H \setminus \{ 0 \}\), and thus \( Y = H/(\lambda r) \in A_H \). Then
\begin{equation*}
    X= \set{g/r}{g \in G}=  \set{ \lambda g /\lambda r}{g \in G} \subseteq \set{h/\lambda r}{h \in H} =  H/(\lambda r)=Y.
\end{equation*}

Conversely, pick any \( G/r \in A_G \) and let \(H/s \in A_H \) be such that \( G/r \subseteq H/s \). Without loss of generality we may assume that \(r\) and \(s\) are positive. Then for every \( g \in G \) we have \( \frac{s}{r} g \in H \). Since \( \frac{s}{r} > 0 \), the map \( g \mapsto \frac{s}{r} g \) is thus an order preserving homomorphism from \( G \) to \( H \).
\end{proof}

\begin{prop}\label{prop : embed reducible to 2nd jump}
\( {\embeds_\mathsf{ArGp}} \leq_B {=_\R^{2\text{-}\mathrm{cf}}} \), thus \( {\equiv_\mathsf{ArGp}} \leq_B E_{=_\R^{2\text{-}\mathrm{cf}}} \).
\end{prop}

\begin{proof}
Given \( G \in X_\mathsf{ArGp} \), construe the set \( A_{\vec{x}_G} \) from equation~\eqref{eq : AxG} as an element of \( (\R^\N)^\N \), that is, fix the enumeration  canonically induced by the sequence \( \vec{x}_G \) of all countable sets involved in the definition of \( A_{\vec{x}_G} \). By Lemma~\ref{lem : emb}, the resulting map is the desired Borel reduction.
\end{proof}


\subsection{Irreducibility results} \label{subsec : irreducibility for embeddability}

As above we use the complete invariants $G\mapsto A_G=\set{G/r}{r\in G\setminus\{0\}}$ for order isomorphism on countable subgroups $G\subseteq\mathbb{R}$, but restricting to a subspace \( \mathcal{X} \) of \( \mathcal{A} \) where the notions of order isomorphism and bi-embeddability coincide.

\begin{prop}
\label{prop : inv embed}
Suppose that $G, H \leq \R$ are bi-embeddable but not isomorphic. Then for any $X\in A_G$ there is $X'\in A_G$ with $X\subsetneq X'$.
\end{prop}

\begin{proof}
Note that $A_G=\set{X/x}{x\in X\setminus\{0\}}$ for any $X\in A_G$.
Thus if $X\in A_G$, $Y\in A_H$, and $X=Y$, then $A_G=A_H$, and therefore $G$ and $H$ are order isomorphic by Proposition~\ref{prop : invariants}.
By Lemma~\ref{lem : emb} and the fact that \( G \) and \( H \) are bi-embeddable, for any $X\in A_G$ there is $Y\in A_H$ with $X\subseteq Y$, and for such \( Y \) there is $X'\in A_G$ with $Y\subseteq X'$.
Since $G$ and $H$ are not isomorphic, it must be that $X\subsetneq Y\subsetneq X'$, so $X\subsetneq X'$.
\end{proof}

Let $\mathcal{X} \subseteq \mathcal{A}$ be the space of all (countable) subgroups $G$ of $\mathbb{R}$ satisfying the following property: for any $r,q\in G \setminus \{ 0 \}$
\[
\text{if } G/r\subseteq G/q \text{ then } G/r=G/q.
\]
Then \( \cong_\mathcal{A} \) and $\equiv_\mathcal{A}$ coincide on $\mathcal{X}$ by Proposition~\ref{prop : inv embed}. In particular, the map $G \mapsto A_G$ is a complete classification of $\equiv_\mathcal{X}$, the restriction of \( \equiv_\mathcal{A} \) to \( \mathcal{X} \).

Consider again the forcing \( \mathbb{P} \) from Definition~\ref{def : forcinggroup}, let $F$ be $\mathbb{P}$-generic over $V(A^1)$, $F'\subset\mathbb{R}$ be the union of all $G_p$ for $p\in F$, and define $\tilde{G}=\bigcup_{r'\in F'}\mathbb{Q}\cdot r'$. That is, $\tilde{G}$ is the divisible hull of the group $G$ from Section~\ref{subsec : complicatedinvariant}.

\begin{lemma}
For any \( r,q \in \tilde{G} \), if \( \tilde{G}/r \subseteq \tilde{G}/q \) then \( \tilde{G}/r = \tilde{G}/q \).
\end{lemma}

\begin{proof}
Set \( X = \tilde{G}/r \) and \( Y = \tilde{G}/q \).
Fix $r',q' \in F'$ such that $r\in \mathbb{Q}\cdot r'$ and $q\in\mathbb{Q}\cdot q'$. 
It follows from the definition of $\tilde{G}$ that $X=\tilde{G}/r'$ and $Y=\tilde{G}/q'$.
If $r'=q'$ then $X=Y$. So it is enough to show that if 
 $r'\neq q'$ then $X\not\subseteq Y$.

Let $\bar{a}$, $\bar{b}$ be finite subsets of $A^1$ such that $r'\in D(\bar{a})$ and $q'\in D(\bar{b})$.
Note that for any $d\in A^1\setminus(\bar{a}\cup\bar{b})$, the two elements $d\cdot r'$, $d\cdot q'$ are distinct.
By density, we may find $d\in A^1$ such that $d\cdot r' \in F'$ and $d\cdot q' \not\in F'$.
It follows that $d\in X$ and $d\notin Y$, as desired.
\end{proof}

This shows that (any enumeration, in some further generic extension, of) \( \tilde{G} \) belongs to \( \mathcal{X} \), so that
 $\tilde{A} \coloneqq A_{\tilde{G}}=\set{\tilde{G}/r}{r\in\tilde{G}\setminus\{0\}}$ is an invariant with respect to  $\equiv_\mathcal{X}$. 

\begin{prop}\label{prop : bi-embed non reduction}
${\equiv_\mathcal{X}} \not\leq_B {\cong^\ast_{3,0}}$, and therefore 
${\equiv_{\mathsf{ArGp}}} \not\leq_B{\cong^\ast_{3,0}}$. 
\end{prop}

\begin{proof}
By the contrapositive of Corollary~\ref{cor : irred to cong}(2),
it suffices to show that if $B\in V(\tilde{A})$ is a countable (in \( V(\tilde{A}) \)) set of sets of reals, definable from $\tilde{A}$ and parameters  in \( V \) alone, then $\tilde{G}\notin V(B)$, and therefore $V(\tilde{A})\neq V(B)$. For this, it is enough to repeat  
the proof of Theorem~\ref{thm : main},
but replacing $A$ with $\tilde{A}$ and $G$ with $\tilde{G}$.
\end{proof}

\begin{cor}
\( {\embeds_\mathsf{ArGp}} \not\leq_B {=_\R^\mathrm{cf}} \).
\end{cor}
\begin{proof}
Since \({=^+}\leq_B {\cong^\ast_{3,0}}\), it follows from Proposition~\ref{prop : bi-embed non reduction} that \({\equiv_{\mathrm{ArGp}}}\not\leq_B {=^+}\). As $=^+$ is equal to \( E_{=_\R^\mathrm{cf}}\), the symmetrization of \(=_\R^\mathrm{cf}\), it follows that \( {\embeds_\mathsf{ArGp}} \not\leq_B {=_\R^\mathrm{cf}} \).
\end{proof}

Finally, we show that $E_{=_\R^{2\text{-}\mathrm{cf}}}$, the equivalence relation canonically induced by $=_\R^{2\text{-}\mathrm{cf}}$, is not Borel reducible to $\equiv_{\mathrm{ArGp}}$. 
This will easily follow from the next result.
\begin{lemma}\label{claim: =++ reduces to E=2cf}
${=_\R^{++}} \leq_B E_{=_\R^{2\text{-}\mathrm{cf}}}$.
\end{lemma}

In order to prove Lemma~\ref{claim: =++ reduces to E=2cf}, we recall the following analysis of the second Friedman-Stanley jump from \cite[Section 3.5]{Sha-thesis}. When considering $=_{\R}^{++}$ as an equivalence relation on $(\mathbb{R}^{\N})^\N$, as we get by applying the Friedman-Stanley jump twice to $=_\R$, the corresponding product topology turns out to be inadequate. We will instead use the following alternative presentation of the second Friedman-Stanley jump.

Consider elements $x\in\R^\N$ as representing the set of reals $S_x=\set{x(i)}{i\in\N}$.
Given $x \in \R^\N$ and some $u\in 2^\N$, we think of $u$ as representing the subset of $x$ defined by $S_x(u)=\set{x(i)}{u(i)=1}$.
Similarly, given some $y = (y_i)_{i \in \N } \in(2^\N)^\N$ we may think of $y$ as a code for a set of subsets of $S_x$, namely \( S_{(x,y)}=\set{S_x(y_i)}{i\in\N} \).
Equip $\R^\N\times (2^\N)^\N$ with the natural product topology, and 
let $F$ be the equivalence relation on $\R^\N\times (2^\N)^\N$ defined by $(x,y)\mathrel{F}(x',y')$ if and only if $S_{(x,y)}=S_{(x',y')}$. Just as \( =_\R^{++} \), 
the relation $F$ is defined to admit a natural complete classification with countable sets of countable sets of reals as complete invariants. Indeed, the relation
$F$ 
can be seen as the pullback of $=_{\R}^{++}$ under the continuous map $\R^\N\times(2^\N)^\N\to (\R^\N)^\N$ sending $(x,y)$ to the element of $(\mathbb{R}^\N)^\N$ whose columns enumerate the sets of reals $S_x(y_0),S_x(y_1), \dots$ in the obvious way.
Conversely, there is a natural Borel reduction from \( =_\R^{++} \) to \( F \), so that $F \sim_B {=_\R^{++}}$ (see~\cite[Claim 3.5.2]{Sha-thesis}).

The reason for considering \( F \) instead of \( =_\R^{++} \) is the following. There is a natural comeager $D\subset(\R^\N)^\N$ such that ${=_\R^{++}\restriction D} \sim_B  {=_\R^+}$ (\cite[Proposition 3.5.1]{Sha-thesis}), namely the set $D$ of all $x = (x_n)_{n \in \N} \in(\R^\N)^\N$ such that $x_n(k)\neq x_m(l)$ for any two distinct pairs $(n,k),(m,l) \in \N^2$. This means that on a topologically large set, the complexity of \( =_\R^{++} \) drops down. In contrast:
\begin{thm}[{\cite[Section 3.5]{Sha-thesis}}]\label{thm: baire category for =++}
For any comeager set $D\subset\R^\N\times(2^\N)^\N$,  the restriction $F\restriction D$ of \( F \) to \( D \) is Borel bi-reducible with $F$ (and so with $=_\R^{++}$).
\end{thm}


In the same fashion, we may consider the following alternative presentation of \( E_{=_\R^{2 \text{-}\mathrm{cf}}} \). Let $E$ be the equivalence relation on $\R^\N\times(2^\N)^\N$ defined by $(x,y)\mathrel{E}(x',y')$ if for any $i \in \N$ there is a $j\in \N$ with $S_x(y_i)\subset S_{x'}(y'_j)$ and for any $i \in \N$ there is a $j \in N$ with $S_{x'}(y'_i)\subset S_x(y_j)$. It is clear that  $E \sim_B E_{=_{\mathbb{R}}^{2-\mathrm{cf}}}$. We are now ready to prove Lemma~\ref{claim: =++ reduces to E=2cf}.

\begin{proof}[Proof of Lemma~\ref{claim: =++ reduces to E=2cf}]
The identity map is clearly a homomorphism from $F$ to $E$, that is, \( (x,y) \mathrel{F} (x',y') \implies (x,y) \mathrel{E} (x',y') \) for all \( (x,y), (x',y') \in  \R^\N\times(2^\N)^\N \). Suppose now that $(x,y)$ and $(x',y')$ are $E$-related but not $F$-related.
Without loss of generality, there is some $X \in S_{(x,y)}$ which is not equal to any set $Y \in S_{(x',y')}$. But since \( (x,y) \mathrel{E} (x',y') \), there are  $Y \in S_{(x',y')}$ and $X' \in S_{(x,y)}$ such that \( X \subseteq Y \subseteq X' \). Since $X\neq Y$, we must have that $X\subsetneq X'$.

Consider now the set $D\subset\R^\N\times(2^\N)^\N$ of all pairs $(x,y)$ such that  $S_x(y_i)\not\subset S_x(y_j)$ for all $i\neq j$.
The discussion above entails that the identity map is a (Borel) reduction of $F\restriction D$ to $E$.
Furthermore, $D$ is comeager in $\subset\R^\N\times(2^\N)^\N$, so by Theorem~\ref{thm: baire category for =++}, 
\[ 
{=_\R^{++}} \sim_B F \leq_B (F\restriction D) \leq_B E \sim_B E_{=_{\mathbb{R}}^{2\text{-}\mathrm{cf}}}. \qedhere
\]
\end{proof}

\begin{prop}
$E_{=_\R^{2\text{-}\mathrm{cf}}} \not\leq_B {\equiv_{\mathrm{ArGp}}}$, thus 
 $=_\R^{2\text{-}\mathrm{cf}} \not\leq_B {\embeds_\mathsf{ArGp}}$.
\end{prop}

\begin{proof}
Arguing as in Proposition~\ref{prop : reduction to cong_3,1} one easily sees that \( \embeds_\mathsf{ArGp} \) is a \( \boldsymbol{\Sigma}^0_4 \) quasi-order, hence \( \equiv_\mathsf{ArGp} \) is a \( \boldsymbol{\Sigma}^0_4\) equivalence relation.
By the results of Hjorth-Kechris-Louveau \cite{HjoKecLou}, the relation $=_{\R}^{++}$ is not Borel reducible to any $\boldsymbol{\Sigma}^0_4$ equivalence relation. In particular \( {=_\R^{++}} \not\leq_B {\equiv_\mathsf{ArGp}} \), thus also $E_{=_\R^{2\text{-}\mathrm{cf}}}\not\leq_B {\equiv_\mathsf{ArGp}}$ by  Lemma~\ref{claim: =++ reduces to E=2cf}.
\end{proof}


\subsection{Another lower bound for \(\equiv_\mathsf{ArGp}\)}\label{subsec: another lower bound}
In an earlier version of this paper we left open the question of whether \(\equiv_\mathsf{ArGp}\) is Borel reducible to an orbit equivalence relation, that is, an equivalence relation induced by a continuous action of a Polish group (see~\cite[Section~3.4]{Gao09}). This problem was since settled by the following argument due to Shani, reducing the maximal \(K_\sigma\) equivalence relation\footnote{A subset of a Polish space is \(K_\sigma\) if it is a countable union of compact sets.} to \(\equiv_\mathsf{ArGp}\).

Consider the space \(\mathcal{P}(\N)^\N\) of all sequences of subsets of \(\N\), with the product topology, where \(\mathcal{P}(\N)\) is identified with the Cantor space \(2^\N\).
Let \(C\subset \mathcal{P}(\N)^\N\) be the set of all \(\seqq{A_n}{n\in\N}\) so that \(A_{n}\subset A_{n+1}\) for each \(n\in\N\).
\begin{defn}
Define \(\leq_{K_\sigma}\) on \(C\) by 
\begin{equation*}
 x\leq_{K_\sigma}y \quad\textrm{ if and only if }\quad\exists l\, \forall m\, (x(m)\subset y(m+l))   
\end{equation*}
Define \(E_{K_\sigma}\) to be \(E_{\leq_{K_\sigma}}\), the induced equivalence relation.
\end{defn}

\begin{thm}[Kechris; Louveau-Rosendal, see~{\cite[Theorem~17]{Ros}}]
The relation $\leq_{K_\sigma}$ is a complete $K_\sigma$ quasi-order, that is: any $K_\sigma$ quasi-order is Borel reducible to $\leq_{K_\sigma}$.
In particular, $E_{{K_\sigma}}$ is a complete $K_\sigma$ equivalence relation.
\end{thm}

\begin{thm}\label{thm : EKsigma red to embed}
${\leq_{K_\sigma}}\leq_B {\sqsubseteq_{\mathrm{ArGp}}}$, and so
${E_{{K_\sigma}}}\leq_B{\equiv_{\mathrm{ArGp}}}$.
\end{thm}
\begin{proof}
Let $\geq_{K_\sigma}$ be the reverse of $\leq_{K_\sigma}$. By the maximality of $\leq_{K_\sigma}$ it follows that ${\geq_{K_\sigma}}\leq_B{\leq_{K_\sigma}}$, and so ${\leq_{K_\sigma}}\leq_B{\geq_{K_\sigma}}$. It therefore suffices to find a Borel reduction from $\geq_{K_\sigma}$ to $\sqsubseteq_{\mathrm{ArGp}}$. We fix some parameters towards the definition.

Fix a sequence $a_0,a_1,...$ of reals with the following property:
\begin{enumerate}[label={\upshape (\( \dagger \))},leftmargin=2pc]
\item \label{dagger}
For any $\alpha\in\set{a_i}{i\in\N}\cup \set{a_i\cdot a_j}{i,j\in\N}$, the divisible additive groups generated by $\{\alpha\}$ and by $\{1\}\cup\set{a_i}{i\in\N}\cup \set{a_i\cdot a_j}{i,j\in\N}\setminus\{\alpha\}$ are disjoint. 
\end{enumerate} 
For example, if $a_0,a_1,...$ are distinct members of the set $Y$ in Fact~\ref{fact : perfect} then they satisfy~\ref{dagger}.
Fix $x\in C$. For $n\in\N$ define $k=k(n,x)\in\N$ to be the minimal $k$ for which $n\in x(k)$ (if exists) and $k(n,x)=\infty$ otherwise.
Let $G^x$ be the additive subgroup of $(\R,+)$ generated by
\begin{equation*}
\{1\}\cup\set{a_n/2^{k(n,x)}}{k(n,x)<\infty}\cup\set{a_n/2^l}{k(n,x)=\infty,\, l\in\N}.    
\end{equation*}
Recall that we may view the domain of $\sqsubseteq_{\mathrm{ArGp}}$ as $\mathcal{A}$, the space of countable sequences enumerating a subgroup of $(\R,+)$.
Fix a Borel map $f\colon C\to \mathcal{A}$ sending $x$ to an enumeration of $G^x$.
We show that $f$ is the desired reduction of $\geq_{K_\sigma}$ to $\sqsubseteq_{\mathrm{ArGp}}$.

Fix $x,y\in C$ and suppose first that $y\leq_{K_\sigma}x$, so there is an $l\in\N$ such that for all $m\in\N$, $y(m)\subset x(m+l)$. Then $\bigcup_m y(m)\subset \bigcup_m x(m)$ and for any $n\in\bigcup_m y(m)$, $k(n,x)\leq k(n,y)+l$. It follows that the map $r\mapsto 2^l\cdot r$ maps $G^x$ into $G^y$, so $f(x)\sqsubseteq_{\mathrm{ArGp}}f(y)$. 

For the other direction, the following claim will be useful.
\begin{claim}
Let $G$ be the additive subgroup of $\R$ generated by $\set{a_n/2^k}{n,k\in\N}\cup\{1\}$. Let $H$ be a subgroup of $G$ so that $1\in H$ and there is some $n,k$ for which $a_n/2^k\in H$. Assume that $\lambda\in \R^+$ is such that the map $r\mapsto \lambda\cdot r$ sends $H$ into $G$. Then $\lambda$ is an integer.
\end{claim}
\begin{proof}[Proof of the claim]
Assume for a contradiction that $\lambda$ is not an integer. Since $1\in H$ it follows that $\lambda\in G$, and so we may write $\lambda=d+\sum_{i=1}^m d_i\cdot a_{n_i}/2^{k_i}$, where $d,d_1,...,d_m$ are integers, $d_i$ is non-zero, and $n_1,...,n_m$ are distinct.

We conclude that $\lambda\cdot a_n/2^k\in G$. It follows that $d_1\cdot a_n\cdot a_{n_1}$ is in the divisible additive group generated by $\{1\}\cup\set{a_i}{i\in\N}\cup\set{a_n\cdot a_{n_i}}{i=2,...,m}$, a contradiction to~\ref{dagger}.
\end{proof}
Assume now that $y\not\leq_{K_\sigma}x$. We need to show that $G^x$ does not embed into $G^y$. Assume towards a contradiction that there is an embedding from $G^x$ into $G^y$. By Hion's lemma, and the above lemma, this embedding must be of the form $r\mapsto \lambda\cdot r$ where $\lambda$ is an integer.
By assumption, for each $l$ there is some $m$ so that $y(m)$ is not contained in $x(m+l)$. So for each $l$ there is some $n$ for which $k(n,x)>k(n,y)+l$ (and in particular $k(n,y)$ is finite). Fix such $l$ and $n$. Then $\lambda\cdot a_n/2^{k(n,x)}\in G^y$. 
By the claim we know that $\lambda$ is an integer, so by~\ref{dagger} it must be that $\lambda\cdot a_n/2^{k(n,x)}=d\cdot a_n/2^{k(n,y)}$ for some integer $d$, and so $\lambda $ is divisible by $2^l$. As this is the case for every $l\in\N$, we reached a contradiction.
\end{proof}
As $E_{K_\sigma}$ is a maximal $K_\sigma$ equivalence relation, it is quite complex. In particular, it follows from a result of Kechris and Louveau~\cite{KL97} that \(E_{K_\sigma}\) is not Borel reducible to any orbit equivalence relation. This is because the equivalence relation $E_1$ is Borel reducible to $E_{K_\sigma}$: see for example the discussion in~\cite[p.4844]{LouRos}. See also the tables in \cite[p.349]{Gao09} and \cite[p.68]{Kano08} (these tables have the equivalence relation \(l^\infty\), which is known to be Borel bi-reducible with \(E_{K_\sigma}\) by the work of Rosendal~\cite{Ros}).  

\begin{cor}\label{cor: biembed not orbit}
\(\equiv_{\mathrm{ArGp}}\) is not Borel reducible to any orbit equivalence relation. In particular \(\equiv_{\mathrm{ArGp}}\) is not classifiable by countable structures, and so it is not Borel reducible to \(\cong_\mathsf{ArGp}\).
\end{cor}
\begin{question}
Is \(\cong_\mathsf{ArGp}\) Borel reducible to \(\equiv_\mathsf{ArGp}\)?
\end{question}

As a bi-product of Theorem~\ref{thm : EKsigma red to embed} and Proposition~\ref{prop : embed reducible to 2nd jump}, we conclude a lower bound for the second Rosendal jump: that \({\leq_{K_\sigma}}\) is Borel reducible to \({=_\R^{2-\mathrm{cf}}}\). The proofs in fact give the following stronger conclusion.
\begin{cor}
The quasi-order \({\leq_{K_\sigma}}\) is Borel reducible to \({=_\N^{2-\mathrm{cf}}}\), and so \({E_{K_\sigma}}\leq_B{E_{=^{2-\mathrm{cf}}_\N}}\).
\end{cor}
\begin{proof}

Given a countable additive subgroup \(M\) of \(\R\), let \(\mathcal{A}_M\) be the subspace of all \(x\in \mathcal{A}\) which enumerate a subgroup of \(M\). Let \(\sqsubseteq_{\mathrm{ArGp}}\restriction \mathcal{A}_M\) be the restriction of \(\sqsubseteq_{\mathrm{ArGp}}\) to \(\mathcal{A}_M\).
Note that the reduction given in Theorem~\ref{thm : EKsigma red to embed} is in fact a Borel reduction from \(\leq_{K_\sigma}\) to \(\sqsubseteq_{\mathrm{ArGp}}\restriction \mathcal{A}_M\), where \(M\) is the additive subgroup of \(\R\) generated by \(\{1\}\cup\set{a_n/2^l}{l, n\in\N}\).

Finally, the reduction given in Proposition~\ref{prop : embed reducible to 2nd jump}, when restricted to \(\mathcal{A}_M\), is a reduction of \(\sqsubseteq_{\mathrm{ArGp}}\restriction \mathcal{A}_M\) to  \(=_M^{2-\mathrm{cf}}\), the second Rosendal jump applied to the equality relation on \(M\). Since \(M\) is countable, \(=_M\) and \(=_\N\) are Borel bi-reducible, and therefore so are \(=_M^{2-\mathrm{cf}}\) and \(=_\N^{2-\mathrm{cf}}\). We conclude that \(\leq_{K_\sigma}\) is Borel reducible to \(=_\N^{2-\mathrm{cf}}\).
\end{proof}
This shows that the iterated Rosendal jump operation over \( =_\N \) (or over any Borel quasi-order containing an infinite antichain) leaves the realm of orbit equivalence relations as soon as possible. Indeed, for an equivalence relation \(F\), the symmetrization of the first Rosendal jump \(E_{F^{\mathrm{cf}}}\) is equal to the Friedman-Stanley jump \(F^+\), and \(F^+\) is Borel reducible to an orbit equivalence relation, whenever \(F\) is.


\subsection{Potential classes of \( \embeds_\mathsf{ArGp} \) and \( \equiv_{\mathsf{ArGp}} \)}

The notion of potential class from Definition~\ref{def : potential classes} straightforwardly adapts to arbitrary binary relations on standard Borel spaces.
However, 
unlike the case for isomorphism relations where Hjorth, Kechris, and Louveau \cite{HjoKecLou} completely classified the possible potential complexities, 
no such theory of potential classes for Borel quasi-orders has been developed yet. Nevertheless, we are able to directly compute 
the (optimal) potential classes of \( \embeds_\mathsf{ArGp} \) and \( \equiv_{\mathsf{ArGp}} \). (Compare the next result with the fact that, as observed after Proposition~\ref{prop : reduction to cong_3,1}, the  potential class of \( \cong_{\mathsf{ArGp}} \) is $\mathrm{D}(\boldsymbol{\Pi}^0_3)$ as well.)

\begin{prop}
The potential class of both \( \embeds_\mathsf{ArGp} \) and \( \equiv_\mathsf{ArGp} \) is $\mathrm{D}(\boldsymbol{\Pi}^0_3)$.
\end{prop}

\begin{proof}
Consider the following representation of $\embeds_\mathsf{ArGp}$, analogous to the one given for $=_\R^{++}$ in Section~\ref{subsec : irreducibility for embeddability}.
Define a Borel map $f$ sending $z\in \mathcal{A}$ to a pair $(x,y)\in \R^{\N}\times 2^{\N\times\N}$ such that $x$ enumerates the field generated by $G$
and $y(n,m)=1$ if and only if \( z(m) \neq 0 \) and $x(n)=z(k)/z(m)$ for some $k$.
Define a relation $R$ on $\mathbb{R}^{\N}\times2^{\N\times\N}$ by $(x,y)\mathrel{R}(x',y')$ if 
\begin{itemize}
    \item $x\mathrel{=_\R^{\mathrm{cf}}}x'$ and
    \item $\exists m,l \,  \forall n,k (x(n)=x'(k)\wedge y(n,m)=1\implies y'(k,l)=1)$.
\end{itemize}
The first is a $\boldsymbol{\Pi}^0_3$ statement, and the second is $\bold{\Sigma}^0_3$, so $R$ is $\mathrm{D}(\boldsymbol{\Pi}^0_3)$.
Finally, it follows from Lemma~\ref{lem : emb} that $f$ is a reduction from $\embeds_\mathsf{ArGp}$ to $R$, which gives that \( \embeds_\mathcal{A} \), and hence also \( \embeds_\mathsf{ArGp} \), is potentially \( \mathrm{D}(\boldsymbol{\Pi}^0_3) \). By symmetrization, it follows that \( \equiv_\mathsf{ArGp} \), is potentially \( \mathrm{D}(\boldsymbol{\Pi}^0_3) \) as well.

To check optimality, we again consider the space \(\mathcal{X} \subseteq \mathcal{A} \)  from Section~\ref{subsec : irreducibility for embeddability} consisting of all (codes for) countable subgroups $G$ of $\mathbb{R}$ 
such that \( G/r\nsubseteq G/q \) whenever \( G/r \neq G/q \). It is easy to check that \( \mathcal{X} \) is a Borel subset of \( \mathcal{A} \) which is also \( \cong \)-invariant, i.e.\ closed under isomorphism. Since bi-embeddability and isomorphism coincide on \( \mathcal{X} \) by Proposition~\ref{prop : inv embed}, it follows that \( \equiv_\mathcal{X} \) is (Borel bi-reducible with) an isomorphism relation, and thus we can apply Hjorth-Kechris-Louveau's theory of potential classes to it. In particular, by Proposition~\ref{prop : bi-embed non reduction} and Proposition~\ref{prop : HjoKecLou} it follows that \( \equiv_\mathcal{X} \) is not potentially simpler than \( \mathrm{D}(\boldsymbol{\Pi}^0_3) \). Thus the same applies to the whole \( {\equiv_\mathcal{A}} \sim_B {\equiv_\mathsf{ArGp}} \), which in turn implies that our computation of the potential class of \( \embeds_\mathsf{ArGp} \) is optimal as well.
\end{proof}

\section{Circular orders}
\label{sec : circular orders}

A \emph{circular order} on a set \(X\) is defined by a \emph{cyclic orientation cocycle}, i.e., a function \(c \colon X^{3} \to \{\pm1, 0\}\) satisfying:
\begin{enumerate}
\item
\(c^{-1}(0) = \Delta(X)\), where \(\Delta(X) \coloneqq \{(x_{1},x_{2},x_{3}) \in X^{3} : x_{i} = x_{j}, \text{ for some \(i\neq j\)}\}\),
\item \(c(x_2,x_3,x_4)-c(x_1,x_3,x_4)+c(x_1,x_2,x_4)-c(x_1,x_2,x_3) = 0\) for
all \(x_1,x_2,x_3,x_4 \in X\).
\end{enumerate}

\begin{defn}
A group \(G\) is \emph{circularly orderable} if it admits a circular order \(c\) which is left-invariant in the sense that \(c(g_1, g_2, g_3) = c(hg_1, hg_2, hg_3)\) for all \(g_1, g_2, g_3, h \in G\).
\end{defn}

Given circularly ordered groups \((G,c)\) and \((H,d)\), an \emph{order preserving homomorphism} is a group homomorphism \(\phi \colon G \to H\) such that \[c(g_1, g_2, g_3) = d(\phi(g_1), \phi(g_2), \phi(g_3))\] for all \(g_1, g_2, g_3 \in G\). It follows from the definition of circular orders that every order preserving homomorphism is automatically an embedding.

\begin{defn}
A circularly ordered group \((G, c)\) is said to be \emph{Archimedean} if there are no elements \(g, h \in G\) such that \(c(1_G,g^n,h) = 1\) for all \(n \geq 1\).
\end{defn}

A well-known example of Archimedean circularly ordered group is \( S^1\) with the obvious circular order induced by the orientation, which can be defined as follows:

\[
c(x,y,z) =\begin{cases}
1& \text{if \((x, y, z)\) is anti-clockwise oriented};\\
-1& \text{if \((x, y, z)\) is clockwise oriented};\\
0 & \text{if any two of  \(x\), \(y\), or \(z\) agree}.
\end{cases}
\]

In fact, Archimidean circularly ordered groups are characterized as follows: 

\begin{prop}[E.g., see~{\cite[Corollary~2.11]{ClaManRiv}}]
Every Archimedean circularly ordered group is isomorphic to a subgroup of \(S^{1}\) via an order preserving isomorphism.
\end{prop}

Let us recall the following construction, which plays an important role also in the work of Clay, Mann, and Rivas~\cite{ClaManRiv} and Bell, Clay, and Ghaswala~\cite{BelClaGha}.

\begin{defn}[{\v{Z}eleva~\cite{Zhe}}]
Given a circularly ordered group \((G, c)\), let \((\widetilde{G},<)\)
 be the central extension of \(G\) by \(\mathbb{Z}\), defined as the set \(G \times \mathbb{Z} \), endowed with the multiplication \((a, n)(b, m) = (ab, n + m + \epsilon_{a,b})\), where
 \[
  \epsilon_{a,b} = 
 \begin{cases} 
0 &\text{if \(a=1_G\) or \(b=1_G\) or \(c(1_G,a,ab)=1\),}\\
1 & \text{if \(ab=1_G\text{ with }a\neq 1_G\) or \(c(1_G,ab,a)=1\)}.
 \end{cases}
 \]
 Given \((a,m), (b,n) \in G\times \Z\) we stipulate
 \[(a,m)<(b,n)\quad\iff\quad m<n\text{ or }(m = n\text{ and }c(1_G, a, b)=1).\]
\end{defn}

 Throughout this section, we are interested in \v{Z}eleva's construction only when \(G\) is a subgroup of \(S^1\). Note that if \(g \in S^1\) with \(g = e^{2\pi i\alpha}\), then the inverse of \((g,0)\) is \((g^{-1}, -1)\) and \((g,0)^n = (g^n, \lfloor n\alpha \rfloor)\), where \(\lfloor \cdot \rfloor\) denotes the floor function.

 As Clay, Mann, and Rivas~\cite[Remark~2.6]{ClaManRiv} pointed out, \v{Z}eleva's construction respects several order theoretic features, such as the Archimedean type property. We highlight the following:
 
 \begin{fact} \label{fact : zeleva}
  If \((G_i, c_i)\), \(i = 1,2\), are circularly ordered groups, and \(\phi \colon G_1 \to G_2\) is an order preserving homomorphism, then the map \(\tilde \phi\colon \tilde G_1 \to \tilde G_2, (g,n)\mapsto(\phi(g),n)\) is an order preserving homomorphism.
 \end{fact}


As we could not find a reference for the following result, we have included a short proof of it using \v{Z}eleva's construction.

\begin{prop}
\label{prop : CO trivial embedding}
Let \((G,c)\) be a subgroup of \(S^1\) with the induced circular order. There are no order preserving embedding of \(G\) into \(S^1\) besides the trivial one.
\end{prop}
\begin{proof}
Assume that \(\iota\colon G \to S^1\) is an order preserving embedding.
Consider any nonzero \(g\in G\) and set \(h= \iota(g)\). We claim that \( h = g \).

The embedding \(\iota\) induces an order preserving isomorphism \(\phi\colon \langle g\rangle \to \langle h\rangle\).
If \(g\) has order \(2\), then \(g = h = e^{i\pi}\). Otherwise, assume towards contradiction that \(g \neq h\).  We write \(g = e^{2\pi i \alpha}\) and \(h = e^{2\pi i \beta}\) for some \(\alpha, \beta \in [0,1)\),
and without loss of generality we assume \(c(1,g,h) = 1\), so that \(\alpha< \beta\).
Let \(n = \min\{m\in \mathbb{N} : \lfloor m\alpha\rfloor < \lfloor m\beta\rfloor\}\) and set \(k = \lfloor n\alpha\rfloor\), so that \( \lfloor n \beta \rfloor = k+1 \).
Using the fact that the map \(\tilde \phi\colon \langle g\rangle \times \mathbb{Z} \to \langle h\rangle \times \mathbb{Z}, (f,\ell) \mapsto (\phi(f),\ell)\) is an order preserving homomorphism by Fact~\ref{fact : zeleva}, we obtain the following identities:
\begin{multline*}
    (h^n,k+1) = (h,0)^n = (\tilde \phi((g,0))^n = \tilde \phi((g,0)^n)  \\
= \tilde \phi((g^n,k)) = (\phi(g^n), k) = ((\phi(g))^n,k) =  (h^n, k).
\end{multline*}
This is clearly a contradiction. We conclude that every \(g\in G\) is fixed by \(\iota\) as desired.
\end{proof}

One can as usual define the space of countable Archimedean circularly ordered groups \(X_\mathsf{CArGp}\) as a space of countable structures with domain \(\N\). We denote by \(\cong_\mathsf{CArGp}\) and \(\equiv_\mathsf{CArGp}\) the isomorphism and bi-embeddability relation on \(X_\mathsf{CArGp}\), respectively.

As for (linearly ordered) Archimedean groups, a more convenient approach for our purpose is considering the subspace of the Polish space \((S^1)^\N\) of countable sequences in \(S^1\). More precisely,
let
\[
\mathcal{C}\coloneqq \set{(x_n)_{n\in \N} \in (S^1)^{(\N)}}{ (x_n)_{n\in \N}\text{ injectively enumerates a subgroup of \(S^1\)}}.
\]
Denote by \(\cong_\mathcal{C}\) and \(\equiv_\mathcal{C}\) the isomorphism and bi-embeddability relation on \(\mathcal{C}\), respectively. Looking at the standard characterization of Archimedean circularly ordered group as subgroups of \(S^1\) (e.g., see\cite[Corollary~2.11]{ClaManRiv}) one can show that \(\cong_\mathsf{CArGp}\) (respectively, \(\equiv_\mathsf{CArGp}\)) and \(\cong_\mathcal{C}\) (respectively, \(\equiv_\mathcal{C}\)) are Borel bi-reducible.

\begin{proof}[Proof of Theorem~\ref{thm : CO}]
We work with the coding space \( \mathcal{C} \).
By Proposition~\ref{prop : CO trivial embedding}, two subgroups of \(S^1\) are bi-embeddable if and only if they are equal. Consequently, isomorphism and bi-embeddability coincide, and  \(\cong_\mathcal{C}\) is Borel reducible to \(=_\R^{+}\): indeed, the map sending \((x_n)_{n \in \N}\) to \((f(x_n))_{n \in \N}\), where \(f\) is any Borel isomorphism between \(S^1\) and \(\mathbb{R}\), is the desired Borel reduction.

For the converse, let \(T\subseteq(0,1)\) be a perfect set of real numbers that are algebraic independent over \(\Q(\pi)\) and fix a Borel bijection \(f\colon \R\to T\). (The existence of such \(T\) follows from \cite{Myc64}.) Then, a Borel reduction from \( =^+_\R \) to \( \cong_\mathcal{C} \) is given by the map \((x_n)_{n\in \N}\in \R^\N\mapsto \langle e^{2\pi i f(x_n)} : n\in \N \rangle\).
\end{proof}

\section{ODAG and other o-minimal theories} \label{sec : ODAG}

A (model-theoretic) type \(p\) is \emph{nonsimple} if for some set \(A\) of realizations of \(p\), there is \(b\notin A\) which realizes \(p\) and which is \(A\)-definable.
Mayer~\cite{May} introduced nonsimple
types in the proof of Vaught's Conjecture for o-minimal theories.
Further, Rast and Sahota~\cite{RasSah} showed that for every o-minimal theory admitting nonsimple types the isomorphism relation \(\cong_T\) on the space of countable models of \(T\) is \(S_\infty\)-complete. That is, every equivalence relation induced by a Borel action of \(S_\infty\) on a standard Borel space (equivalently: every isomorphism relation) is Borel reducible to~\(\cong_T\), and thus the latter is as complicated as possible. This is proved by showing that under such hypothesis on \( T \), the relation \( \cong_{\mathsf{LO}} \) of isomorphism on countable linear orders, which is a well known example of an \( S_\infty \)-complete equivalence relation, is Borel reducible to \( \cong_T \).

An important example of an o-minimal theory admitting nonsimple types is the one of \emph{ordered divisible Abelian groups} (\(\mathsf{ODAG}\)) ---  here the unique type at \(+\infty\) is nonsimple. 
 In this section we discuss the proof of Theorem~\ref{theorem : ODAG}, which mirrors the results of Rast and Sahota~\cite{RasSah} to the context of the bi-embeddability relation for ordered divisible Abelian groups.


A quasi-order \(Q\) on a standard Borel space \(X\) is said to be \emph{analytic} (or \( \analytic \)) if it is an analytic subset of \( X^2 \), i.e.\ a projection of some Borel subset of \( X^2 \times Y \) with \( Y \) standard Borel. All embeddability relations mentioned below are analytic quasi-orders, after the appropriate coding. A quasi-order \( Q \) is 
\emph{complete \(\analytic\)} if whenever \(P\) is an analytic quasi-order, then \(P \leq_B Q\). Similarly, an analytic equivalence relation \(E\) on a Polish space \(X\) is said to be \emph{complete \(\analytic\)} if whenever \(F\) is an analytic equivalence relation on a Polish space \(Y\), then \(F \leq_B E\). In particular, in both cases completeness implies that the relation is not Borel. If \(Q\) is a complete \(\analytic\) quasi-order on a Polish space \(X\), then \(E_{Q}\) 
is complete \(\analytic\). Conversely, by Louveau and Rosendal~\cite[Proposition~1.5]{LouRos} every complete \(\analytic\) equivalence relation can be obtained by symmetrization from a complete \(\analytic\) quasi-order.

A natural attempt to prove Theorem~\ref{theorem : ODAG} would be to adapt Rast-Sahota's proof to show that the embeddability relation \(\embeds_\mathsf{LO}\) on countable linear orders  is
Borel reducible to the relation \( \embeds_{\mathsf{ODAG}} \) of embeddability on countable ordered divisible Abelian groups. Unfortunately, it turns out that \( \embeds_{\mathsf{LO}} \) is not a complete \( \analytic \) quasi-order. Indeed, Laver~\cite{Lav71} proved that \(\embeds_\mathsf{LO}\) is a better quasi-order, and thus combinatorially simpler than the arbitrary \( \analytic \) quasi-order. Moreover, the same obstruction implies that the associated bi-embeddability relation \(\equiv_\mathsf{LO}\) cannot be a complete \(\analytic\) equivalence relation because even a very simple equivalence relation such as \( =_\R \) is not Borel reducible to \(\equiv_\mathsf{LO}\) (see e.g.\ \cite[Lemma 3.17]{CamMarMot}).

Luckily, the embeddability relation for \emph{colored} countable linear orders is an example of complete \(\analytic\) quasi-order. This is defined as follows.
A \emph{colored linear order} on \(\N\) is a pair \(L=(<_{L}, c_{L})\) such that \(<_{L}\) is a strict linear order on \(\N\) and \(c_{L}\colon \N \to {\N}\).
Denote by \(X_\mathsf{CLO}\) the Polish space of all colored linear orders on \(\N\).

Given \({K,L} \in X_\mathsf{CLO}\) we say that \(K\) \emph{is embeddable into} \(L\)
(in symbols \(K\embeds_\mathsf{CLO} L\)) if there exists \(f\colon \N \to \N\) such that:
\begin{enumerate}
\item
\(m <_{K} n\) implies \(f(m) <_{L} f(n)\) for every \(m,n\in \N\);
\item
\( c_{L}(f(n)) = c_{K}(n)\) for every \(n\in \N\).
\end{enumerate}
Notice that by anti-reflexivity and linearity of strict orders, condition (1) implies that \( f \) is injective and that \(m <_{K} n \iff f(m) <_{L} f(n)\) for every \(m,n\in \N\).
Thus, broadly speaking, a colored linear order embeds into another one if there is an order-and-color preserving injection from the former into the latter.
It is clear by the definition that \(\embeds_\mathsf{CLO}\)  is a \(\analytic\) quasi-order on \( X_\mathsf{CLO}\).

The following result was observed by Louveau (see Marcone and Rosendal~\cite[Theorem~3.2]{MarRos}).

\begin{thm}
\label{theorem : MarRos}
The embeddability relation on colored linear orders \(\embeds_\mathsf{CLO}\) is a complete \(\analytic\) quasi-order.
\end{thm}

The proof of Theorem~\ref{theorem : ODAG} follows from Theorem~\ref{theorem : MarRos} and the following proposition.

\begin{prop}\label{prop : CLO and ODAG}
There is a Borel reduction from \(\embeds_\mathsf{CLO}\) to the embeddability relation \(\embeds_\mathsf{ODAG}\) on the space of ordered divisible Abelian groups.
\end{prop}

\begin{proof}
Let \(\set{H_{n}}{n\in \N}\) be a family of pairwise non-embeddable countable divisible (Archimedean) subgroups of \(\R\). (See Remark~\ref{rk : incomparable}.)
For every \(L=(<_{L}, c_L)\in X_\mathsf{CLO}\), we construct the ordered divisible Abelian group \(G_{L}\)
 as the group of finite support functions \(f\) with domain \( \N \) and such that \(f(n)\in H_{c(n)}\). 

More precisely, \(G_L\) is the subgroup of \(\prod_{n \in L} H_{c_L(n)}\) (where the group operation is computed componentwise) consisting of those \(f\) whose support
\[
S_f = \set{n\in L}{ f(n) \neq 0}
\]
is finite. For later use, when \(f \in G_L\) we set \(v(f) = \max_{<_L} S_f\).
Further, let \(<_{G_L}\) be the reverse lexicographic order on \(G_{L}\). That is, for every distinct \(f, g \in G_{L}\), set \(f <_{G_L} g\) if and only if \(f(m) <_{H_m} g(m)\) where \(m = \max_{<_L} \set{n \in \N }{f(n)\neq g(n) }\). (Clearly, \( m \in S_{f} \cup S_{g} \), and thus it is well defined.)

For every positive \(f, g \in G_L\) we set \(f\preceq_L g\) if and only if there exists some \(n \in \N\) such that \(f \leq ng \), and \(f \approx_L g\) if \(f\preceq_L g\) and \(g\preceq_L f\).
Notice that \(f \preceq_L g\) if and only if \(v(f) \leq_{L} v(g)\). The \( \approx_L \)-equivalence classes are exactly the maximal Archimedean subgroups of \( G_L \), and if \( H \) is an Archimedean subgroup of \( G_L \) then \( f \approx_L g \) for all positive \( f,g \in H \), so that \(  v(f) = v(g) \).

The map \(L\mapsto G_L\) is easily seen to be Borel, and we are going to show that it is also a reduction from \(\embeds_\mathsf{CLO}\) to \(\embeds_\mathsf{ODAG}\). 

If \(j\colon \N\to \N\) is a color preserving embedding between
\(K=(<_K,c_K)\) and \(L=(<_L, c_L)\), then \(H_{c_L(j(n))}= H_{c_K(n)}\)
for every \(n\in \N\).
Thus, the map sending \(f\in G_K\) to the element \(g\in G_L\) defined by
\[
g(m) = 
\begin{cases}
f(n) & \text{if \(m = j(n)\) for some \(n\in\N\)}\\
0 & \text{otherwise}
\end{cases}
\]
is well-defined (as by construction \(S_g = j(S_f)\) is finite) and is clearly an order preserving group embedding from \(G_K\) to \(G_L\).

Conversely, suppose that \(\phi\colon G_{K} \to G_{L}\) is an an order preserving homomorphism, and notice that \( \phi \) reduces \( \approx_K \) to \( \approx_L \), that is \( f \approx_K g \iff \phi(f) \approx_L \phi(g) \) for all \( f,g \in G_K \).
Given \(n \in \N\), consider any positive \(f \in G_K\) with \(S_f = \{n\}\) and set 
\[
\phi^*(n) = v(\phi(f)).
\]
Notice that the definition of \(\phi^*(n)\) does not depend on the choice of \(f\). In fact, \( f \approx_K g \) for any positive \(g \in G_K\) such that \(S_f = S_g = \{n\}\), so that \(\phi(f) \approx_L \phi(g)\), and thus \(v(\phi(f)) = v(\phi(g))\), as desired.

It is easy to show that \(m <_{K} n \iff \phi^{*}(m) <_{L} \phi^{*}(n)\). Indeed, if \( f,g \in G_K \) were used to define \( \phi^*(m) \) and \( \phi^*(n) \), then
\begin{align*}
m <_K n & \iff f <_{G_K} g \text{ and } f \not\approx_K g \\
& \iff \phi(f) <_{G_L} \phi(g) \text{ and } \phi(f) \not\approx_L \phi(g) \\
&  \iff v(\phi(f)) <_L v(\phi(g)) \\
& \iff \phi^{*}(m) <_{L} \phi^{*}(n).
\end{align*}
Next, for every \(n \in \N \) define 
\[
G_K(n) = \set{g\in G_K}{v(g)\leq_K n}\quad\text{and}\quad N_K(n) = \set{g\in G_K}{v(g)<_Kn}.
\] 
It is easily checked that \(N_K(n)\) is a convex subgroup of \(G_K(n)\), thus 
the quotient \(G_K(n) / N_K(n)\) is an ordered group, which is isomorphic to \(H_{c_K(n)}\). Since \(\phi\) is an ordered group embedding, \(\phi\) reduces \( \approx_K \) to \( \approx_L \), and thus it induces an embedding from \(H_{c_K(n)}\) into  \(H_{c_L(\phi^*(n))}\). It follows that \(c_{K}(n) =c_{L}(\phi^*(n))\) for every \(n\), hence  \(\phi^*\colon K \to L\) witnesses \( L \embeds_{\mathsf{CLO}} K \), as desired.
\end{proof}

Another well-known example of o-minimal theory is the one of real closed fields.
Recall that the theory of \emph{real closed fields} (\textsf{RCF}) is the theory in
the language \(\{ + , \cdot , 0 , 1 , < \}\), axiomatized by the ordered field axioms, an axiom saying that every positive element has a square root, and for every odd number \(d\) and every polynomial  \(P(x)\) of degree \(d\), an axiom saying that  \(P(x)\) has a root.

The following corollary is a consequence of Theorem~\ref{theorem : ODAG}.

\begin{cor}\label{corollary : RCF}
The embeddability relation \( \embeds_{\mathsf{RCF}} \) on countable reals closed fields is a complete \( \analytic \) quasi-order. Consequently, the associated
 bi-embeddability relation \(\equiv_{\mathsf{RCF}}\)  is a complete \( \analytic \) equivalence relation.
\end{cor}

\begin{proof}
First recall the following well-known construction.
Suppose that \(\mathbf{k}\subset \R\) is a countable real closed field. For every ordered divisible Abelian group \(G\), let \(\mathbf{k}((t^G))\) be the field of \emph{Hahn series}, that is, 
the field of formal series
 \(f=\sum_{g\in G} a_{g}t^{g}\)
 with \(a_{g}\in \mathbf{k}\), and  well-ordered support \(S_{f}\coloneqq \{g \in G \mid a_{g}\neq 0\}\).  
Addition and multiplication are defined as suggested by the series notation: for example, 
\(t^{g}t^{h}= t^{g+h}\) for all \(g,h\in G\), where addition is computed in \(G\).
We can clearly regard  \(\boldsymbol{k}\) as a subfield of  \(\mathbf{k}((t^G))\) via the identification
\( a = a t^{0}\),
for every \(a \in \boldsymbol{k}\).
Finally, consider the valuation \(v(f) = \min { S_{f} } \) and the ordering obtained by setting \(0 < f\) if and only if \(f(v(f)) > 0\). Then \(\mathbf{k}((t^G))\) is a real closed valued field with value group \(G\) and residue field \(\mathbf{k}\).

We now construct a Borel reduction from \( \embeds_{\mathsf{ODAG}} \) to \( \embeds_{\mathsf{RCF}} \); in view of Theorem~\ref{theorem : ODAG} this gives the desired result.
Fix a countable real closed field \(\mathbf{k}\subset \R\).
For every countable ordered divisible Abelian group \( G \) and every  \(g \in G\), set \( f_g = t^g \), i.e.\ the formal series \(f_{g} \in \mathbf{k}((t^G))\) with coefficients
\[
a_{h} = \begin{cases}
1 & h =g,\\
0 & \text{otherwise}
\end{cases}
\]
for all \( h \in G \). Then define the real closed field \(R_{G}\) as the real closure of the algebraic field extension \(\mathop{\boldsymbol{k}} (f_{g} : g\in G)\).
If \(\phi\colon G\to H\) is an embedding, then we can define an embedding from \(R_{G}\) into \(R_{H}\) sending
\(t^{g}\) to \(t^{\phi(g)}\) for all \(g\in G\) and then extending such map in the obvious way.
Conversely, since any embedding between real closed fields induces an embedding between the corresponding value groups, if \(R_{G}\) embeds into \( R_{H}\) then there is an order preserving homomorphism \(G \to H\).
\end{proof}

Theorem~\ref{theorem : ODAG} and Corollary~\ref{corollary : RCF} give two examples of an o-minimal theory \(T\) admitting nonsimple types  and such that the corresponding relations of embeddability \( \embeds_T \) and bi-embeddability \(\equiv_T\) on countable models of \( T \) both have maximal complexity with respect to Borel reducibility.
Nevertheless, we cannot extend such result to all such theories.

\begin{ex}
Let \(T\) be the theory of \((\mathbb{Q}, <, s)\) where \(s\) is the unary function \(x\mapsto x+1\). 
The countable models of \(T\) are exactly the structures
  \[M_L=(L\times \Q, \prec_L, s_L),\] where \(L\) is a countable linear order, \(\prec_L\) is the lexicographic order on \(L\times \mathbb{Q}\), and \(s_L\colon (l,q)\mapsto (l,q+1)\).  Moreover, \(M_L\) is embeddable into
\(M_{K}\) if and only if \(L\) is embeddable into \(K\), therefore \( {\embeds_{T}} \sim_B {\embeds_{\mathsf{LO}}} \) and \( \embeds_T \) is a well-quasi order. Thus \( {\equiv_{T} } \sim_B {\equiv_\mathsf{LO}}\), and hence \( \equiv_T \) is not a complete \(\analytic\) equivalence relation.
\end{ex}

\begin{question}
For which o-minimal theories is  the bi-embeddability relation \(\equiv_T\) complete \(\analytic\)?
\end{question}


\begin{thebibliography}{CMMR18}

\bibitem[AK00]{AdaKec}
Scot Adams and Alexander~S. Kechris.
\newblock Linear algebraic groups and countable {B}orel equivalence relations.
\newblock {\em J. Amer. Math. Soc.}, 13(4):909--943, 2000.

\bibitem[Bae37]{Bae}
Reinhold Baer.
\newblock Abelian groups without elements of finite order.
\newblock {\em Duke Math. J.}, 3(1):68--122, 1937.

\bibitem[BCG]{BelClaGha}
Jason Bell, Adam Clay, and Tyrone Ghaswala.
\newblock Promoting circular-orderability to left-orderability.
\newblock {\em Annal. Inst. math. Fourier}, 71 (1): 175--201, 2021.

\bibitem[BS18]{BaiSam}
Hyungryul Baik and Eric Samperton.
\newblock Spaces of invariant circular orders of groups.
\newblock {\em Groups Geom. Dyn.}, 12(2):721--763, 2018.

\bibitem[CC22]{CC22}
Filippo Calderoni and Adam Clay.
\newblock Borel structures on the space of left-ordering.
\newblock {\em Bull. London Math. Soc.},  54(1): 83 -- 94, 2022.

\bibitem[CD03]{CalDun}
Danny Calegari and Nathan~M. Dunfield.
\newblock Laminations and groups of homeomorphisms of the circle.
\newblock {\em Invent. Math.}, 152(1):149--204, 2003.

\bibitem[CG19]{ClaGha}
Adam Clay and Tyrone Ghaswala.
\newblock Free products of circularly ordered groups with amalgamated subgroup.
\newblock {\em J. London Math. Soc. (2)}, 100(3):775--803, 2019.

\bibitem[Cla12]{Cla12}
Adam Clay.
\newblock Free lattice-ordered groups and the space of left orderings.
\newblock {\em Monatsh. Math.}, 167(3-4):417--430, 2012.

\bibitem[CMMR18]{CamMarMot}
Riccardo Camerlo, Alberto Marcone, and Luca Motto~Ros.
\newblock On isometry and isometric embeddability between ultrametric {P}olish
  spaces.
\newblock {\em Adv. Math.}, 329:1231--1284, 2018.

\bibitem[CMR18]{ClaManRiv}
Adam Clay, Kathryn Mann, and Crist\'{o}bal Rivas.
\newblock On the number of circular orders on a group.
\newblock {\em J. Algebra}, 504:336--363, 2018.

\bibitem[CR16]{ClaRol}
Adam Clay and Dale Rolfsen.
\newblock {\em Ordered groups and topology}, volume 176 of {\em Graduate
  Studies in Mathematics}.
\newblock American Mathematical Society, Providence, RI, 2016.

\bibitem[DH21]{DH21}
Natasha Dobrinen and Dan Hathaway.
\newblock Classes of barren extention.
\newblock {\em J. Symbolic Logic},  86 (1): 178--209, 2021.

\bibitem[DM08]{DowMon}
Rod Downey and Antonio Montalb\'{a}n.
\newblock The isomorphism problem for torsion-free abelian groups is analytic
  complete.
\newblock {\em J. Algebra}, 320(6):2291--2300, 2008.

\bibitem[Eff65]{Eff65}
Edward~G. Effros.
\newblock Transformation groups and {$C^{\ast} $}-algebras.
\newblock {\em Ann. of Math. (2)}, 81:38--55, 1965.

\bibitem[Fel71]{Fel71}
Ulrich Felgner.
\newblock {\em Models of {${\rm ZF}$}-set theory}.
\newblock Lecture Notes in Mathematics, Vol. 223. Springer-Verlag, Berlin-New
  York, 1971.

\bibitem[For00]{For}
Matthew Foreman.
\newblock A descriptive view of ergodic theory.
\newblock In {\em Descriptive set theory and dynamical systems
  ({M}arseille-{L}uminy, 1996)}, volume 277 of {\em London Math. Soc. Lecture
  Note Ser.}, pages 87--171. Cambridge Univ. Press, Cambridge, 2000.

\bibitem[For18]{For18}
Matthew Foreman.
\newblock What is a {B}orel reduction?
\newblock {\em Notices Amer. Math. Soc.}, 65(10):1263--1268, 2018.

\bibitem[FRW11]{ForRudWei}
Matthew Foreman, Daniel~J. Rudolph, and Benjamin Weiss.
\newblock The conjugacy problem in ergodic theory.
\newblock {\em Ann. of Math. (2)}, 173(3):1529--1586, 2011.

\bibitem[FS89]{FriSta}
Harvey Friedman and Lee Stanley.
\newblock A {B}orel reducibility theory for classes of countable structures.
\newblock {\em J. Symbolic Logic}, 54(3):894--914, 1989.

\bibitem[FW04]{ForWei}
Matthew Foreman and Benjamin Weiss.
\newblock An anti-classification theorem for ergodic measure preserving
  transformations.
\newblock {\em J. Eur. Math. Soc. (JEMS)}, 6(3):277--292, 2004.

\bibitem[Gao09]{Gao09} S. Gao, Invariant descriptive set theory, CRC Press, 2009.

\bibitem[Ghy01]{Ghy}
\'{E}tienne Ghys.
\newblock Groups acting on the circle.
\newblock {\em Enseign. Math. (2)}, 47(3-4):329--407, 2001.

\bibitem[Gli61]{Gli61}
James Glimm.
\newblock Locally compact transformation groups.
\newblock {\em Trans. Amer. Math. Soc.}, 101:124--138, 1961.

\bibitem[Hio54]{Hio54}
Ya.~V. Hion.
\newblock Archimedean ordered rings.
\newblock {\em Uspehi Mat. Nauk (N.S.)}, 9(4(62)):237--242, 1954.

\bibitem[Hjo00]{Hjo00}
Greg Hjorth.
\newblock {\em Classification and orbit equivalence relations}, volume~75 of
  {\em Mathematical Surveys and Monographs}.
\newblock American Mathematical Society, Providence, RI, 2000.

\bibitem[Hjo02]{Hjo02}
Greg Hjorth.
\newblock The isomorphism relation on countable torsion free abelian groups.
\newblock {\em Fund. Math.}, 175(3):241--257, 2002.

\bibitem[HKL90]{HarKecLou}
L.~A. Harrington, A.~S. Kechris, and A.~Louveau.
\newblock A {G}limm-{E}ffros dichotomy for {B}orel equivalence relations.
\newblock {\em J. Amer. Math. Soc.}, 3(4):903--928, 1990.

\bibitem[HKL98]{HjoKecLou}
Greg Hjorth, Alexander~S. Kechris, and Alain Louveau.
\newblock Borel equivalence relations induced by actions of the symmetric
  group.
\newblock {\em Ann. Pure Appl. Logic}, 92(1):63--112, 1998.

\bibitem[HMW85]{HMW}
J.~M. Henle, A.~R.~D. Mathias, and W.~Hugh Woodin.
\newblock A barren extension.
\newblock In {\em Methods in mathematical logic ({C}aracas, 1983)}, volume 1130
  of {\em Lecture Notes in Math.}, pages 195--207. Springer, Berlin, 1985.
  
  \bibitem[Hol01]{Hol}
Otto H\"{o}lder.
\newblock Die axiome der quantit\"{a}t und die lehre vom mas.
\newblock {\em Ber. Verh. Sachs. Ges. Wiss. Leipzig Math. Phys. Cl.}, 53:1--64,
  1901.

\bibitem[Jec71]{Jec71}
Tom\'{a}\v{s} Jech.
\newblock On models for set theory without {${\rm AC}$}.
\newblock In {\em Axiomatic {S}et {T}heory ({P}roc. {S}ympos. {P}ure {M}ath.,
  {V}ol. {XIII}, {P}art {I}, {U}niv. {C}alifornia, {L}os {A}ngeles, {C}alif.,
  1967)}, pages 135--141. Amer. Math. Soc., Providence, R.I., 1971.

\bibitem[Jec73]{Jec73}
Thomas~J. Jech.
\newblock {\em The axiom of choice}.
\newblock North-Holland Publishing Co., Amsterdam-London; Amercan Elsevier
  Publishing Co., Inc., New York, 1973.
\newblock Studies in Logic and the Foundations of Mathematics, Vol. 75.

\bibitem[Jec03]{Jec03}
Thomas Jech.
\newblock {\em Set theory}.
\newblock Springer Monographs in Mathematics. Springer-Verlag, Berlin, 2003.
\newblock The third millennium edition, revised and expanded.

\bibitem[JKL02]{JacKecLou}
S.~Jackson, A.~S. Kechris, and A.~Louveau.
\newblock Countable {B}orel equivalence relations.
\newblock {\em J. Math. Log.}, 2(1):1--80, 2002.

\bibitem[Kan08]{Kan08}
Akihiro Kanamori.
\newblock Cohen and set theory.
\newblock {\em Bull. Symbolic Logic}, 14(3):351--378, 2008.

\bibitem[Kano08]{Kano08} V. Kanovei, Borel equivalence relations, Amer. Math. Soc., 2008.

\bibitem[Kar18]{Kar18}
Asaf Karagila.
\newblock The {B}ristol model: an abyss called a {C}ohen real.
\newblock {\em J. Math. Log.}, 18(2):1850008, 37, 2018.

\bibitem[Kar19]{Kar19}
Asaf Karagila.
\newblock Iterating symmetric extensions.
\newblock {\em J. Symb. Log.}, 84(1):123--159, 2019.

\bibitem[Kec95]{Kec}
Alexander~S. Kechris.
\newblock {\em Classical descriptive set theory}, volume 156 of {\em Graduate
  Texts in Mathematics}.
\newblock Springer-Verlag, New York, 1995.

\bibitem[Kec99]{Kec99}
Alexander~S. Kechris.
\newblock New directions in descriptive set theory.
\newblock {\em Bull. Symbolic Logic}, 5(2):161--174, 1999.

\bibitem[KL97]{KL97} A. S. Kechris and A. Louveau, The classification of hypersmooth Borel equivalence relations, Journal of the American Mathematical Society, vol. 10 (1997), no. 1, pp. 215-242.

\bibitem[Lav71]{Lav71}
Richard Laver.
\newblock On {F}ra\"{\i}ss\'{e}'s order type conjecture.
\newblock {\em Ann. of Math. (2)}, 93:89--111, 1971.

\bibitem[Lin11]{Lin11}
Peter~A. Linnell.
\newblock The space of left orders of a group is either finite or uncountable.
\newblock {\em Bull. Lond. Math. Soc.}, 43(1):200--202, 2011.

\bibitem[Lou94]{Lou94}
Alain Louveau.
\newblock On the reducibility order between {B}orel equivalence relations.
\newblock In {\em Logic, methodology and philosophy of science, {IX}
  ({U}ppsala, 1991)}, volume 134 of {\em Stud. Logic Found. Math.}, pages
  151--155. North-Holland, Amsterdam, 1994.

\bibitem[LR05]{LouRos}
Alain Louveau and Christian Rosendal.
\newblock Complete analytic equivalence relations.
\newblock {\em Trans. Amer. Math. Soc.}, 357(12):4839--4866, 2005.

\bibitem[Man15]{Man15}
Kathryn Mann.
\newblock Left-orderable groups that don't act on the line.
\newblock {\em Math. Z.}, 280(3-4):905--918, 2015.

\bibitem[May88]{May}
Laura~L. Mayer.
\newblock Vaught's conjecture for o-minimal theories.
\newblock {\em J. Symbolic Logic}, 53(1):146--159, 1988.

\bibitem[Min72]{Min72}
Donald~P. Minassian.
\newblock On {T}eh's construction of orders in abelian groups.
\newblock {\em Proc. Cambridge Philos. Soc.}, 71:433--436, 1972.

\bibitem[Mon73]{Mon73}
G.~P. Monro.
\newblock Models of {${\rm ZF}$} with the same sets of sets of ordinals.
\newblock {\em Fund. Math.}, 80(2):105--110, 1973.

\bibitem[MR04]{MarRos}
Alberto Marcone and Christian Rosendal.
\newblock The complexity of continuous embeddability between dendrites.
\newblock {\em J. Symbolic Logic}, 69(3):663--673, 2004.

\bibitem[MR18]{ManRiv}
Kathryn Mann and Crist\'{o}bal Rivas.
\newblock Group orderings, dynamics, and rigidity.
\newblock {\em Ann. Inst. Fourier (Grenoble)}, 68(4):1399--1445, 2018.

\bibitem[Myc64]{Myc64}
Jan Mycielski.
\newblock Independent sets in topological algebras.
\newblock {\em Fund. Math.}, 55:139--147, 1964.

\bibitem[Nav10]{Nav10}
Andr\'{e}s Navas.
\newblock On the dynamics of (left) orderable groups.
\newblock {\em Ann. Inst. Fourier (Grenoble)}, 60(5):1685--1740, 2010.

\bibitem[Riv12]{Riv12}
Crist\'{o}bal Rivas.
\newblock Left-orderings on free products of groups.
\newblock {\em J. Algebra}, 350:318--329, 2012.

\bibitem[Ros05]{Ros}
Christian Rosendal.
\newblock Cofinal families of {B}orel equivalence relations and quasiorders.
\newblock {\em J. Symbolic Logic}, 70(4):1325--1340, 2005.

\bibitem[RSS17]{RasSah}
Richard Rast and Davender Singh~Sahota.
\newblock The {B}orel complexity of isomorphism for o-minimal theories.
\newblock {\em J. Symb. Log.}, 82(2):453--473, 2017.

\bibitem[Sab16]{Sab16}
Marcin Sabok.
\newblock Completeness of the isomorphism problem for separable {$\rm
  C^\ast$}-algebras.
\newblock {\em Invent. Math.}, 204(3):833--868, 2016.

\bibitem[Sha19]{Sha-thesis}
Assaf Shani.
\newblock {\em Borel equivalence relations and symmetric models}.
\newblock PhD thesis, UCLA, 2019.

\bibitem[Sha21]{Sha}
Assaf Shani.
\newblock Borel reducibility and symmetric models.
\newblock {\em Trans. Amer. Math. Soc.}, 374: 453--485: 2021.

\bibitem[Sik04]{Sik04}
Adam~S. Sikora.
\newblock Topology on the spaces of orderings of groups.
\newblock {\em Bull. London Math. Soc.}, 36(4):519--526, 2004.

\bibitem[Sil80]{Sil80}
Jack~H. Silver.
\newblock Counting the number of equivalence classes of {B}orel and coanalytic
  equivalence relations.
\newblock {\em Ann. Math. Logic}, 18(1):1--28, 1980.

\bibitem[Sto91]{Sto}
Otto Stolz.
\newblock Ueber das {A}xiom des {A}rchimedes.
\newblock {\em Math. Ann.}, 39(1):107--112, 1891.

\bibitem[Teh61]{Teh}
H.-H. Teh.
\newblock Construction of orders in {A}belian groups.
\newblock {\em Proc. Cambridge Philos. Soc.}, 57:476--482, 1961.

\bibitem[Tho02]{Tho02}
Simon Thomas.
\newblock On the complexity of the classification problem for torsion-free
  abelian groups of rank two.
\newblock {\em Acta Math.}, 189(2):287--305, 2002.

\bibitem[Tho03]{Tho03}
Simon Thomas.
\newblock The classification problem for torsion-free abelian groups of finite
  rank.
\newblock {\em J. Amer. Math. Soc.}, 16(1):233--258, 2003.

\bibitem[Wag93]{Wag}
Stan Wagon.
\newblock {\em The {B}anach-{T}arski paradox}.
\newblock Cambridge University Press, Cambridge, 1993.
\newblock With a foreword by Jan Mycielski, Corrected reprint of the 1985
  original.
  
  \bibitem[Zel76]{Zhe}
S.~D. \v{Z}eleva.
\newblock Cyclically ordered groups.
\newblock {\em Sibirsk. Mat. \v{Z}.}, 17(5):1046--1051, 1976.
  
\end{thebibliography}
\end{document}